\documentclass[11pt]{article}
\usepackage{amscd}
\usepackage{amsfonts}
\usepackage{amsmath}
\usepackage{amssymb}
\usepackage{amsthm}
\usepackage{bbm}
\usepackage{CJK}
\usepackage{fancyhdr}
\usepackage{graphicx}
\usepackage{indentfirst}
\usepackage{latexsym,bm}
\usepackage{mathrsfs}
\usepackage{xypic}
\usepackage[top=1in,bottom=1in,left=1.25in,right=1.25in]{geometry}
\newtheorem{theorem}{Theorem}[section]
\newtheorem{lemma}[theorem]{Lemma}
\newtheorem{definition}[theorem]{Definition}
\newtheorem{proposition}[theorem]{Proposition}
\newtheorem{example}[theorem]{Example}

\allowdisplaybreaks[4]
\def\<{\langle}
\def\>{\rangle}
\def\a{\alpha}
\def\b{\beta}

\def\l{\lambda}
\def\o{\otimes}
\def\om{\omega}

\def\t{\tau}
\def\tr{\triangleright}

\def\un{\underline}

\date{}
\begin{document}
\renewcommand{\baselinestretch}{1.2}
\renewcommand{\arraystretch}{1.0}
\title{\bf More properties of Yetter-Drinfeld category over dual quasi-Hopf algebras}
 \date{}
\author {{\bf Daowei Lu$^1$,  Xiaohui Zhang$^2$, Dingguo Wang$^2$\footnote{Corresponding author: dgwang@qfnu.edu.cn
 }   }\\
{\small $^1$Department of Mathematics, Jining University}\\
{\small Qufu, Shandong 273155, P. R. China}\\
{\small $^2$ School of Mathematical Sciences, Qufu Normal University}\\
{\small Qufu, Shandong 273165, P. R. of China}}

 \maketitle
\begin{center}
\begin{minipage}{12.cm}

\noindent{\bf Abstract.}
Let $H$ be a dual quasi-Hopf algebra. A. Balan has introduced the right-right Yetter-Drinfeld module over $H$ in [ Yetter-Drinfeld modules and Galois extensions over coquasi-Hopf algebras, U.P.B. Sci. Bull., Series A, 71(2009): 43--60]. In this paper we will introduce more properties of Yetter-Drinfeld modules. Firstly we will give all other three kinds of Yetter-Drinfeld modules over $H$, and the monoidal and braided structure of the categories of Yetter-Drinfeld modules explicitly. Then we prove that the category $^H_H\mathcal{YD}^{fd}$ of finite-dimensional left-left Yetter-Drinfeld modules is rigid. Finally we will study the braided cocommutativity of the coalgebra $H_0$ in $^H_H\mathcal{YD}$.
\\

\\

 \end{minipage}
 \end{center}
 \normalsize\vskip1cm

\section*{Introduction}

Quasi-bialgebras and quasi-Hopf algebras have been introduced by Drinfeld in \cite{D}, in connection with the Knizhnik-Zamolodchikov system of partial differential equations, and have been used afterwards in several branches of mathematics and physics.  In a quasi-bialgebra $H$, the comultiplication is not coassociative but is quasi-coassociative in the sense that the comultiplication is coassociative up to conjugation by an invertible element $\Phi\in H\o H\o H$. Equivalently the representation category of $H$ is not a strict monoidal category while the reassociator is not trivial. If we draw our attention to the category of co-representations of a coalgebra with non-associative multiplication, we get the concepts of dual quasi-bialgebra and of dual quasi-Hopf algebra. These notions have been introduced by Majid in \cite{M2} in order to prove a Tannaka-Krein type theorem for quasi-Hopf algebras.

For a dual quasi-Hopf algebra $H$, the category of right $H$-comodule $\mathcal{M}^H$ is monoidal with the usual tensor product. The difference between a dual quasi-Hopf algebra and a Hopf algebra lies in the fact that the associativity of tensor product in the category $\mathcal{M}^H$ is not trivial, but modified by an invertible element $\sigma\in(H\o H\o H)^*$. Consequently the multiplication of $H$ is no longer associative.

In \cite{M2}, the left Yetter-Drinfeld module over quasi-Hopf algebras was firstly constructed by S. Majid with the help of the isomorphism between the category of Yetter-Drinfeld modules and the center of the representation category. Subsequently, Bulacu, Caenepeel and Panaite in \cite{BCP} introduced all kinds of Yetter-Drinfeld modules and showed that the category of finite dimensional Yetter-Drinfeld modules is rigid. Following the ideas of S. Majid, in \cite{Ba3} Balan introduced the notion of right-right Yetter-Drinfeld module over dual quasi-Hopf algebra and studied its Galois extension. Later on, Ardizzoni in \cite{AP} introduced another form of Yetter-Drinfeld module through the isomorphism between the category of Yetter-Drinfeld modules and the category of Hopf bimodules.

Motivated by these ideas, in this paper, we will continue the study of the category of Yetter-Drinfeld modules over dual quasi-Hopf algebra $H$. Firstly we will give explicitly the other three kinds of Yetter-Drinfeld modules, and their braided monoidal structures. Moreover $H$ itself is made to be a Yetter-Drinfeld module via the coadjoint coaction and quasi-analogue regular action. And some isomorphisms of braided monoidal category are established, such as
$$\mathcal{YD}^H_H\cong\! \overline{ ^{H^{op,cop}} _{H^{op,cop}}\mathcal{YD}},\ \ ^H\mathcal{YD}_H\cong\! \overline{_{H^{op,cop}}\mathcal{YD}^{H^{op,cop}}},$$
and
$$ ^H\mathcal{YD}^{in}_H\cong\! ^H_H\mathcal{YD}.$$
Then we will show that the category $^H_H\mathcal{YD}^{fd}$ of finite-dimensional left-left Yetter-Drinfeld modules is braided rigid. Finally we will study a special object $H_0$, and prove that it is a braided cocommutative coalgebra in the category of Yetter-Drinfeld modules.

This paper is organized as follows. In section 1, we will review the basic results about dual quasi-Hopf algebras and monoidal category. In section 2, we will firstly introduce the notion of left-left Yetter-Drinfeld modules over dual quasi-Hopf algebra $H$. In particular, as in the case of Hopf algebras, $H$ itself is made to be an object in $^H_H\mathcal{YD}$, the category of left-left Yetter-Drinfeld modules. For this purpose, we will introduce some equations needed in our computations. We will give explicitly the braided monoidal structure of the categories of Yetter-Drinfeld modules. Then we will construct some isomorphisms between the Yetter-Drinfeld categories.
In section 3, we will verify that the category of finite dimensional left-left Yetter-Drinfeld modules is rigid, and give the explicit forms of the left and right duals of any object. In section 4, we will firstly show that $H_0$ is a coalgebra in $^H_H\mathcal{YD}$, and then prove that it is braided cocommutative.

\section{Preliminary}
\def\theequation{1.\arabic{equation}}
\setcounter{equation} {0} 

Throughout this article, let $k$ be a fixed field. All algebras, coalgebras, linear spaces etc. will be over $k$; unadorned $\o$ means $\o_k$.

\subsection{Dual quasi-Hopf algebra}

Recall from \cite{M1} that a dual quasi-bialgebra $H$ is a coassociative coalgebra with comultiplication $\Delta$ and counit $\varepsilon$ together with coalgebra morphisms $m_H:H\o H\rightarrow H$ (the multiplication, we write $m_H(h\o h')=hh'$) and $\eta_H:k\rightarrow H$ (the unit, we write $\eta_H(1)=1$), and a invertible element $\sigma\in(H\o H\o H)^*$ (the reassociator), such that for all $a,b,c,d\in H$ the following relations hold
\begin{align}
a_1(b_1c_1)\sigma(a_2,b_2,c_2)&=\sigma(a_1,b_1,c_1)(a_2b_2)c_2,\label{1a}\\
1a&=a1=a,\label{1b}\\
\sigma(a_1,b_1,c_1d_1)\sigma(a_2b_2,c_2,d_2)&=\sigma(b_1,c_1,d_1)\sigma(a_1,b_2c_2,d_2)\sigma(a_2,b_3,c_3),\label{1c}\\
\sigma(a,1,b)&=\varepsilon(a)\varepsilon(b).\label{1d}
\end{align}

$H$ is called a dual quasi-Hopf algebra if, moreover, there exists an anti-morphism $s$ of the coalgebra $H$ and element $\a,\b\in H^*$ such that for all $h\in H$,
\begin{eqnarray}
&&s(h_1)\a(h_2)h_3=\a(h)1,\quad\quad h_1\b(h_2)s(h_3)=\b(h)1,\label{1e}\\
&&\sigma(h_1\b(h_2),s(h_3),\a(h_4)h_5)=\sigma^{-1}(s(h_1),\a(h_2)h_3,\b(h_4)s(h_5))=\varepsilon(h).\label{1f}
\end{eqnarray}

It follows from the axioms that $s(1)=1$ and $\a(1)\b(1)=1$. Moreover (\ref{1c}) and (\ref{1d}) imply that
\begin{equation}
\sigma(1,a,b)=\sigma(a,b,1)=\varepsilon(a)\varepsilon(b).\label{1g}
\end{equation}

Together with a dual quasi-Hopf algebra $H=(H,m,1,\Delta,\varepsilon,\sigma,s,\a,\b)$ with bijective antipode, we also have $H^{op},H^{cop}$ and $H^{op,cop}$ as dual quasi-Hopf algebras. The dual quasi-Hopf structures are obtained by putting
$\sigma_{cop}=\sigma^{-1},\sigma_{op}=(\sigma^{-1})^{321}$, and $\sigma_{op,cop}=\sigma^{321}$. $s_{op}=s_{cop}=(s^{-1}_{op,cop})=s^{-1}$, $\a_{cop}=\b s^{-1}$, $\a_{op}=\a s^{-1}$, $\a_{op,cop}=\b$, $\b_{cop}=\a s^{-1}$, $\b_{op}=\b s^{-1}$, $\b_{op,cop}=\a$. Here $\sigma^{321}(a,b,c)=\sigma(c,b,a)$.

We recall that an invertible element $F\in(H\o H)^*$ satisfying $F(1,a)=F(a,1)=\varepsilon(a)$, induces a twisted transformation
\begin{eqnarray}
&&a\cdot b=F(a_1,b_1)a_2b_2F^{-1}(a_3,b_3),\\
&&\sigma_F(a,b,c)=F(b_1,c_1)F(a_1,b_2c_2)\sigma(a_2,b_3,c_3)F^{-1}(a_3b_4,c_4)F^{-1}(a_4,b_5)\label{1h'}
\end{eqnarray}

For a Hopf algebra, one knows that the antipode is an anti-algebra morphism, i.e., $s(ab)=s(b)s(a)$. For a dual quasi-Hopf algebra, this is true only up to a twist, namely, there exists a twist transformation $f\in(H\o H)^*$ such that for all $a,b\in H$,
\begin{equation}
f(a_1,b_1)s(a_2b_2)g(a_3,b_3)=s(b)s(a),\label{1h}
\end{equation}
where $g$ denotes the convolution inverse of $f$.

The element $f$ can be computed explicitly. For all $a,b,c,d\in H$, set
\begin{eqnarray*}
&&\nu(a,b,c,d)=\sigma(a_1,b_1,c_1)\sigma^{-1}(a_2b_2,c_2,d),\\
&&\mu(a,b,c,d)=\sigma(a_1b_1,c_1,d)\sigma^{-1}(a_2,b_2,c_2).
\end{eqnarray*}
Define elements $\l,\chi\in(H\o H)^*$ by
\begin{eqnarray*}
&&\l(a,b)=\nu(s(b_1),s(a_1),a_3,b_3)\a(a_2)\a(b_2),\\
&&\chi(a,b)=\mu(a_1,b_1,s(b_3),s(a_3))\b(a_2)\b(b_2).
\end{eqnarray*}

Then $f$ and $g$ are given by the following formulae:
\begin{eqnarray*}
&&f(a,b)=\sigma^{-1}(s(b_1)s(a_1),a_3b_3,s(a_5b_5))\l(a_2,b_2)\b(a_4b_4),\\
&&g(a,b)=\sigma^{-1}(s(a_1b_1),a_3b_3,s(b_5)s(a_5))\chi(a_4,b_4)\a(a_2b_2).
\end{eqnarray*}

The elements $\l,\chi$ and the twist $f$ fulfill the relations
\begin{equation}
f(a_1,b_1)\a(a_2b_2)=\l(a,b),\quad \b(a_1b_1)g(a_2,b_2)=\chi(a,b).\label{1i}
\end{equation}

The corresponding reassociator is given by
\begin{equation}
\sigma_f(a,b,c)=\sigma(s(c),s(b),s(a)).\label{1j}
\end{equation}

Recall from \cite{L} that a coquasitriangular dual quasi-Hopf algebra is a dual quasi-Hopf algebra $H$ with an invertible element $\varphi\in(H\o H)^*$ satisfying
\begin{align*}
&(1)~\varphi(gh,l)=\sigma(l_1,g_1,h_1)\varphi(g_2,l_2)\sigma^{-1}(g_3,l_3,h_2)\varphi(h_3,l_4)\sigma(g_4,h_4,l_5),\\
&(2)~\varphi(g,hl)=\sigma^{-1}(h_1, l_1,g_1)\varphi(g_2,l_2)\sigma(h_2,g_3,l_3)\varphi(g_4,h_3)\sigma^{-1}(g_5,h_4,l_4),\\
&(3)~\varphi(g_1,h_1)g_2h_2=h_1 g_1\varphi(g_2,h_2),\\
&(4)~\varphi(1,g)=\varphi(g,1)=\varepsilon(g),
\end{align*}
for all $g,h,l\in H$.

\subsection{Monoidal categories and Center construction}

A monoidal category means a category $\mathcal{C}$ with objects $U,V,W,$ etc., a functor $\o:\mathcal{C}\times\mathcal{C}\rightarrow\mathcal{C}$ equipped with an natural transformation consisting of functorial isomorphism $a_{_{U,V,W}}:(U\o V)\o W\rightarrow U\o(V\o W)$ satisfying a pentagon identity, and a compatible unit object $I$ and associated functorial isomorphisms (the left and the right unit constraints, $l_V:V\cong V\o I$ and $r_V:V\cong I\o V$, respectively.) Now if $\mathcal{C}$ and $\mathcal{D}$ are monoidal categories then, roughly speaking, we say that $F:\mathcal{C}\rightarrow\mathcal{D}$ is a monoidal functor if it respects the tensor products (in the sense that for any two objects $U,V\in\mathcal{C}$ there exists a functorial isomorphism $\Psi:F(U)\o F(V)\rightarrow F(U\o V)$ such that $\Psi$ respects the associativity constraints), the unit object and the left and right unit
constraints (for a complete definition see \cite{M1}).

If $H$ is a dual quasi-Hopf algebra, then the categories $\mathcal{M}^H$ and $\mathcal{^HM}$ are monoidal categories. The associative constraint on $\mathcal{M}^H$ is the following: for any $M,N,P\in\mathcal{M}^H$, and $m\in M,n\in N$,
$a_{_{M,N,P}}:(M\o N)\o P\rightarrow M\o(N\o P)$ is given by
\begin{eqnarray*}
a_{_{M,N,P}}((m\o n)\o p)=\sigma(m_{(1)},n_{(1)},p_{(1)})m_{(0)}\o(n_{(0)}\o p_{(0)})
\end{eqnarray*}
On $\mathcal{^HM}$, the associative constraint is given by
\begin{eqnarray*}
a_{_{M,N,P}}((m\o n)\o p)=\sigma^{-1}(m_{(-1)},n_{(-1)},p_{(-1)})m_{(0)}\o(n_{(0)}\o p_{(0)}).
\end{eqnarray*}

Let $(\mathcal{C},\o,I,a,l,r)$ be a monoidal category, and $V\in\mathcal{C}$. $V^*\in\mathcal{C}$ is called the left dual of $V$, if there exist two morphisms $ev_{_{V}}:V^*\o V\rightarrow I$ and $coev_{_{V}}:I\rightarrow V\o V^*$ such that
\begin{eqnarray*}
&&(V\o ev_{_{V}})\circ a_{_{V,V^*,V}}\circ(coev_{_{V}}\o V)=V,\\
&&(ev_{_{V}}\o V^*)\circ a^{-1}_{_{V^*,V,V^*}}\circ(V^*\o coev_{_{V}})=V^*.
\end{eqnarray*}
$^*V\in\mathcal{C}$ is called a right dual of $V$ if there exist two morphisms $ev'_{_{V}}:V\o\! ^*V\rightarrow I$ and $coev'_{_{V}}:I\rightarrow\! ^*V\o V$ such that
\begin{eqnarray*}
&&(^*V\o ev'_{_{V}})\circ a_{_{^*V,V,^*V}}\circ(coev'_{_{V}}\o\ ^*V)=\ ^*V,\\
&&(ev'_{_{V}}\o V)\circ a^{-1}_{_{V,^*V,V}}\circ(V\o coev'_{_{V}})=V.
\end{eqnarray*}
$\mathcal{C}$ is called a rigid monoidal category if every object of $\mathcal{C}$ has a left and right dual. The category $\mathcal{^HM}_{fd}$ of finite dimensional modules over a dual quasi-Hopf algebra $H$ is rigid. For $V\in\mathcal{^HM}_{fd}$, $V^*=$Hom$(V,k)$ with left coaction $\l(\varphi)=\langle \varphi,v_{i(0)}\rangle s^{-1}(v_{i(-1)})\o v^i$. The evaluation and coevaluation are given by
\begin{equation}
ev_{_V}(\varphi\o v)=\b(s^{-1}(v_{(-1)}))\varphi(v_{(0)}),\quad coev_{_V}(1)=\a(s^{-1}(v_{i(-1)}))v_{i(0)}\o v^i.
\end{equation}
where $\{v_i\}_i$ is a basis in $V$ with dual basis $\{v^i\}_i$.

The right dual $^*V$ of $V$ is the same dual vector space equipped with the left $H$-comodule structure given by $\l(\varphi)=\langle \varphi,v_{i(0)}\rangle s(v_{i(-1)})\o v^i$ and
\begin{equation}
ev'_{_V}(v\o\varphi)=\b(v_{(-1)})\varphi(v_{(0)}),\quad coev'_{_V}(1)=v^i\o\a(v_{i(-1)})v_{i(0)}.
\end{equation}

For a braided monoidal category $\mathcal{C}$, let $\mathcal{C}^{in}$ be equal to $\mathcal{C}$ as a monoidal category, with the mirror-reversed braiding $\tilde{c}_{_{M,N}}=c^{-1}_{_{M,N}}$.

Following \cite{M2}, the left weak center $\mathcal{W}_l\mathcal{(C)}$ is the category with the objects $(V,s_{_{V,-}})$, where $V\in\mathcal{C}$ and $s_{_{V,-}}:V\o-\rightarrow-\o V$ is a family of natural transforms such that $s_{_{V,I}}=id_V$ and for all $X,Y\in\mathcal{C}$
\begin{equation}
(X\o s_{_{V,Y}})\circ a_{_{X,V,Y}}\circ(s_{_{V,X}}\o Y)=a_{_{X,Y,V}}\circ s_{_{V,X\o Y}}\circ a_{_{V,X,Y}}.
\end{equation}
A morphism between $(V,s_{_{V,-}})$ and $(V',s_{_{V',-}})$ consists of $\psi:V\rightarrow V'$ in $\mathcal{C}$ such that
$$(X\o \psi)\circ s_{_{V,X}}=c_{_{V',X}}\circ(\psi\o X).$$
$\mathcal{W}_l\mathcal{(C)}$ is a prebraided monoidal category. The tensor product is
$$(V,s_{_{V,-}})\o(V',s_{_{V',-}})=(V\o V',s_{_{V\o V',-}}),$$
with
\begin{equation}
s_{_{V\o V',X}}=a_{_{X,V,V'}}\circ(s_{_{V,X}}\o V')\circ a^{-1}_{_{V,X,V'}}\circ(V\o s_{_{V',X}})\circ a_{_{V,V',X}},\label{1n}
\end{equation}
and the unit is $(I,id)$. The braiding $s$ on $\mathcal{W}_l\mathcal{(C)}$ is given by
$$c_{_{V,V'}}=s_{_{V,V'}}:(V,s_{_{V,-}})\o(V',s_{_{V',-}})\rightarrow(V',s_{_{V',-}})\o(V,s_{_{V,-}}).$$
The center $\mathcal{Z}_l\mathcal{(C)}$ is the full subcategory of $\mathcal{W}_l\mathcal{(C)}$ consisting of objects $(V,s_{_{V,-}})$ with $s_{_{V,-}}$ a natural isomorphism. $\mathcal{Z}_l\mathcal{(C)}$ is a braided monoidal category.

The right weak center $\mathcal{W}_r\mathcal{(C)}$ is the category with the objects $(V,c_{_{-,V}})$, where $V\in\mathcal{C}$ and $t_{_{-,V}}:-\o V\rightarrow V\o-$ is a family of natural transforms such that $t_{_{I,V}}=id_V$ and
\begin{equation}
a^{-1}_{_{V,X,Y}}\circ t_{_{X\o Y,V}}\circ a^{-1}_{_{X,Y,V}}=(t_{_{X,V}}\o Y)\circ a^{-1}_{_{X,V,Y}}\circ(X\o t_{_{Y,V}}),\label{1k}
\end{equation}
for all $X,Y\in\mathcal{C}$. A morphism between $(V,t_{_{-,V}})$ and $(V',t_{_{-,V'}})$ consists of $\psi:V\rightarrow V'$ in $\mathcal{C}$ such that
$$(\psi\o X)\circ t_{_{X,V}}=t_{_{X,V'}}\circ(X\o \psi).$$

$\mathcal{W}_r\mathcal{(C)}$ is a prebraided monoidal category. The unit is $(I,id)$ and the tensor product is
$$(V,t_{_{-,V}})\o(V',t_{_{-,V'}})=(V\o V',(V,t_{_{-,V\o V'}}))$$
with
\begin{equation}
t_{_{-,V\o V'}}=a^{-1}_{_{V,V',X}}\circ(V\o t_{_{X,V'}})\circ a_{_{V,X,V'}}\circ(t_{_{X,V}}\o V')\circ a^{-1}_{_{X, V,V'}}.\label{1m}
\end{equation}
The braiding $d$ is given by
$$d_{_{V,V'}}=t_{_{V,V'}}:(V,s_{_{-,V}})\o(V',s_{_{-,V'}})\rightarrow(V',s_{_{-,V'}})\o(V,s_{_{-,V}}).$$

The center $\mathcal{Z}_r\mathcal{(C)}$ is the full subcategory of $\mathcal{W}_r\mathcal{(C)}$ consisting of objects $(V,t_{_{-,V}})$ with $t_{_{-,V}}$ a natural isomorphism. $\mathcal{Z}_r\mathcal{(C)}$ is a braided monoidal category.

Let $(\mathcal{C},\o,I,a,l,r)$ be a monoidal category. Then we have a second monoidal structure on $\mathcal{C}$, defined as
$$\overline{\mathcal{C}}=(\mathcal{C},\overline{\o}=\o\circ\t,I,\overline{a},r, l),$$
where $\t:\mathcal{C}\times \mathcal{C}\rightarrow \mathcal{C}\times \mathcal{C}, (U,V)\mapsto(V,U)$ and $\overline{a}$ given by $\overline{a}_{U,V,W}=a^{-1}_{W,V,U}$.

If $c$ is a braiding on $\mathcal{C}$, then $\overline{c}$, defined by $\overline{c}_{U,V}=c_{V,U}$ is a braiding on $\mathcal{C}$.

It is obvious that
\begin{proposition}\cite{BCP}
Let $(\mathcal{C},\o,I,a,l,r)$ be a monoidal category. Then
$$\overline{\mathcal{W}_l(\mathcal{C})}\cong \mathcal{W}_r(\overline{\mathcal{C}}),\ \ \overline{\mathcal{W}_r(\mathcal{C})}\cong \mathcal{W}_l(\overline{\mathcal{C}}),$$
as prebraided monoidal category. And
$$\overline{\mathcal{Z}_l(\mathcal{C})}\cong \mathcal{Z}_r(\overline{\mathcal{C}}),\ \ \overline{\mathcal{Z}_r(\mathcal{C})}\cong \mathcal{Z}_l(\overline{\mathcal{C}}),$$
as braided monoidal category.
\end{proposition}

\section{Yetter-Drinfeld modules over a dual quasi-Hopf algebra}
\def\theequation{2.\arabic{equation}}
\setcounter{equation} {0} \hskip\parindent

In \cite{Ba2}, the category of right-right Yetter-Drinfeld module was constructed. In this section, we will construct all kinds of Yetter-Drinfeld modules over a dual-quasi Hopf algebra, and give their braided monoidal structure.

In a Hopf algebra $H$, we obviously have the identity $h(g_1s(g_2))=h\varepsilon(g)$ for all $g,h\in H$. Now this formula could be generalized to the dual quasi-Hopf algebra setting.

Let $H$ be a dual quasi-Hopf algebra. Recall from \cite{Ba2}, for all $a,b\in H$, define elements $p^R,q^R,p^L,q^L$ in $(H\o H)^*$ by
\begin{eqnarray*}
&&p^R(a,b)=\sigma^{-1}(a,b_1,s(b_3))\b(b_2),\quad q^R(a,b)=\sigma(a,b_3,s^{-1}(b_1))\a(s^{-1}(b_2)),\\
&&p^L(a,b)=\sigma(s^{-1}(a_3),a_1,b)\b(s^{-1}(a_2)),\quad q^L(a,b)=\sigma^{-1}(s(a_1),a_3,b)\a(a_2),
\end{eqnarray*}

\begin{lemma}
Let $H$ be a dual quasi-Hopf algebra with bijective antipode $s$.
For all $a,b\in H$,
\begin{eqnarray}
&&p^R(a_1,b)a_2=(a_1b_1)p^R(a_2,b_2)s(a_3),\quad q^R(a_2,b)a_1=(a_2b_3)q^R(a_1,b_2)s^{-1}(b_1),\label{2a}\\
&&p^L(a,b_1)b_2=s^{-1}(a_3)p^L(a_2,b_2)(a_1b_1),\quad q^L(a,b_2)b_1=s(a_1)q^L(a_2,b_1)(a_3b_2),\label{2b}
\end{eqnarray}
and
\begin{eqnarray}
&&q^R(a_1b_1,s(b_3))p^R(a_2,b_2)=\varepsilon(a)\varepsilon(b),\quad p^L(s(a_1),a_3b_2)q^L(a_2,b_1)=\varepsilon(a)\varepsilon(b),\label{2f}\\
&&q^L(s^{-1}(a_3),a_1b_1)p^L(a_2,b_2)=\varepsilon(a)\varepsilon(b),\ q^R(a_1,b_2)p^R(a_2b_3,s^{-1}(b_1))=\varepsilon(a)\varepsilon(b),\label{2g}\\
&&q^R(s(b_2),s(a_2))\sigma^{-1}(s(b_1)s(a_1),a_3,b_4)\a(b_3)=\lambda(a,b),\\
&&q^R(a_1,s(b_2))(a_2s(b_1))b_3=a_1q^R(a_2,s(b)).
\end{eqnarray}
Moreover we have the following formulae
\begin{align}
(1)\quad&q^R(a_1,b_1)q^R(a_2b_2,c_1)\sigma^{-1}(a_3,b_3,c_2)\nonumber\\
&=\sigma(a_2(b_4c_4),s^{-1}(c_1),s^{-1}(b_1))f(s^{-1}(c_2),s^{-1}(b_2))q^R(a_1,b_3c_3).\label{2c}\\
(2)\quad&\sigma(a_1,b_1,c_1)p^R(a_2b_2,c_2)p^R(a_3,b_3)\nonumber\\
  &=\sigma^{-1}(a_1(b_1c_1),s(c_4),s(b_4))p^R(a_2,b_2c_2)g(b_3,c_3).\label{2s}
  \end{align}
\end{lemma}

\begin{proof}
We only prove (\ref{2c}) and (\ref{2s}), since the proof of (\ref{2a})--(2.5) is an easy exercise.
Define a map $\om:(H^{\o5})^*\rightarrow (H^{\o3})^*$ by
$$\om(\varphi)(a,b,c)=\varphi(a,b_3,c_3,s^{-1}(c_1),s^{-1}(b_1))\a(s^{-1}(b_2))\a(s^{-1}(c_2)),$$
for all $\varphi\in(H^{\o5})^*$ and $a,b,c\in H$.

The left side of (\ref{2c}) is equal to $\om(X)$, where for all $a,b,c,d,e\in H$,
$$X(a,b,c,d,e)=\sigma(a_1,b_1,e)\sigma(a_2b_2,c_1,d)\sigma^{-1}(a_3,b_3,c_2).$$
Since
\begin{align*}
&f(s^{-1}(c_1),s^{-1}(b_1))q^R(a,b_2c_2)\\
&=f(s^{-1}(c_1),s^{-1}(b_1))\sigma(a,b_4c_4,s^{-1}(b_2c_2))\a(s^{-1}(b_3c_3))\\
&\stackrel{(\ref{1h})}{=}\sigma(a,b_4c_4,s^{-1}(c_1)s^{-1}(b_1))f(s^{-1}(c_3),s^{-1}(b_3))\a(s^{-1}(c_2)s^{-1}(b_2))\\
&\stackrel{(\ref{1i})}{=}\sigma(a,b_3c_3,s^{-1}(c_1)s^{-1}(b_1))\l(s^{-1}(c_2),s^{-1}(b_2))\\
&=\sigma(a,b_5c_5,s^{-1}(c_1)s^{-1}(b_1))\nu(b_4,c_4,s^{-1}(b_2),s^{-1}(b_2))\a(s^{-1}(b_3))\a(s^{-1}(c_3)),
\end{align*}
we have the right side of (\ref{2c}):
\begin{align*}
&\sigma(a_2(b_4c_4),s^{-1}(c_1),s^{-1}(b_1))f(s^{-1}(c_2),s^{-1}(b_2))q^R(a_1,b_3c_3)\\
&=\sigma(a_2(b_8c_9),s^{-1}(c_1),s^{-1}(b_1))\sigma(a_1,b_7c_8,s^{-1}(c_2)s^{-1}(b_2))\\
&\quad\sigma(b_5,c_6,s^{-1}(c_4))\sigma^{-1}(b_6c_7,s^{-1}(c_3),s^{-1}(b_3))\a(s^{-1}(b_4))\a(s^{-1}(c_5)).
\end{align*}
Define an element $Y\in(H^{\o5})^*$ by
$$Y(a,b,c,d,e)=\sigma^{-1}(c_1,d_1,e_1)\sigma(b_1,c_2,d_2e_2)\sigma(a_1,b_2c_3,d_3e_3)\sigma(a_2(b_3c_4),d_4,e_4).$$
Then
\begin{align*}
&\om(Y)(a,b,c)\\
&=Y(a,b_3,c_3,s^{-1}(c_1),s^{-1}(b_1))\a(s^{-1}(b_2))\a(s^{-1}(c_2))\\
             &=\sigma^{-1}(c_6,s^{-1}(c_4),s^{-1}(b_4))\sigma(b_6,c_7,s^{-1}(c_3)s^{-1}(b_3))\sigma(a_1,b_7c_8,s^{-1}(c_2)s^{-1}(b_2))\\
             &\quad\sigma(a_2(b_8c_9),s^{-1}(c_1),s^{-1}(b_1))\a(s^{-1}(b_5))\a(s^{-1}(c_5))\\
             &\stackrel{(\ref{1c})(\ref{1e})}{=}\sigma(a_2(b_8c_9),s^{-1}(c_1),s^{-1}(b_1))\sigma(a_1,b_7c_8,s^{-1}(c_2)s^{-1}(b_2))\\
             &\quad\sigma(b_5,c_6,s^{-1}(c_4))\sigma^{-1}(b_6c_7,s^{-1}(c_3),s^{-1}(b_3))\a(s^{-1}(b_4))\a(s^{-1}(c_5)),
\end{align*}
which is equal to the right side of (\ref{2c}).
Moreover
\begin{align*}
&Y(a,b,c,d,e)\\
&\stackrel{(\ref{1c})}{=}\sigma^{-1}(c_1,d_1,e_1)\sigma(a_1,b_1,c_2(d_2e_2))\sigma(a_2b_2,c_3,d_3e_3)
            \sigma^{-1}(a_3,b_3,c_4)\sigma(a_4(b_4c_5),d_4,e_4)\\
            &\stackrel{(\ref{1a})}{=}\sigma(a_1,b_1,(c_1d_1)e_1)\sigma^{-1}(c_2,d_2,e_2)\sigma(a_2b_2,c_3,d_3e_3)
            \sigma((a_3b_3)c_4,d_4,e_4)\sigma^{-1}(a_4,b_4,c_5)\\
            &\stackrel{(\ref{1c})}{=}\sigma(a_1,b_1,(c_1d_1)e_1)\sigma(a_2b_2,c_2d_2,e_2)
            \sigma(a_3b_3,c_3,d_3)\sigma^{-1}(a_4,b_4,c_4).
\end{align*}

Then
\begin{align*}
&\om(Y)(a,b,c)\\
&=\sigma(a_1,b_4,(c_5\a(s^{-1}(c_4))s^{-1}(c_3))s^{-1}(b_2))\sigma(a_2b_5,c_6s^{-1}(c_2),s^{-1}(b_1))\\
             &\quad\sigma(a_3b_6,c_7,s^{-1}(c_1))\sigma^{-1}(a_4,b_7,c_8)\a(s^{-1}(b_3))\\
             &\stackrel{(\ref{1e})}{=}\sigma(a_1,b_3,s^{-1}(b_1))
             \sigma(a_2b_4,c_3,s^{-1}(c_1))\sigma^{-1}(a_3,b_5,c_4)\a(s^{-1}(b_3))\a(s^{-1}(c_2))\\
             &\stackrel{(\ref{1c})(\ref{1e})}{=}\sigma(a_1,b_3,s^{-1}(b_1))\sigma(b_4,c_4,s^{-1}(c_2))\sigma(a_2,b_5c_5,s^{-1}(c_1))\a(s^{-1}(b_2))\a(s^{-1}(c_3))\\
             &\stackrel{(\ref{1e})}{=}\sigma(a_1,b_3,s^{-1}(b_1))\sigma^{-1}(a_2,b_4,c_3s^{-1}(c_3))\sigma(b_5,c_6,s^{-1}(c_2))\\
             &\quad\sigma(a_3,b_6c_7,s^{-1}(c_1))\a(s^{-1}(b_2))\a(s^{-1}(c_4))\\
             &\stackrel{(\ref{1c})}{=}\sigma(a_1,b_3,s^{-1}(b_1))\sigma(a_2b_4,c_3,s^{-1}(c_1))\sigma^{-1}(a_3,b_5,c_4)
             \a(s^{-1}(b_2))\a(s^{-1}(c_2))\\
             &=\om(X)(a,b,c).
\end{align*}
The relation (\ref{2c}) is obtained. Since $H^{cop}$ is also a dual quasi-Hopf algebra, by (\ref{2c}) we have (\ref{2s}).
The proof is completed.
\end{proof}

\begin{lemma}
Define an element $U\in(H\o H)^*$ by
$$U(a,b)=g(a_1,b_1)q^R(s(b_2),s(a_2)),$$
for all $a,b\in H$. Then we have the relation
\begin{eqnarray}
&&(1)~U(a,b_1)s(b_2)=s(a_1b_1)U(a_2,b_2)a_3,\label{2d}\\
&&(2)~U(a,b_1c_1)U(b_2,c_2)=\sigma(a_1,b_1,c_1)\sigma(s((a_2b_2)c_2),a_4,c_4)U(a_3b_3,c_3),\label{2e}\\
&&(3)~p^L(s(a_1b_1),a_3)U(a_2,b_2)=p^R(a,b).\label{2z}
\end{eqnarray}
\end{lemma}

\begin{proof}
The verification of (\ref{2d}) is straightforward. We only prove (\ref{2e}) and (\ref{2z}).
For all $a,b,c\in H$,
\begin{align*}
&\sigma^{-1}(a_1,b_1,c_1)U(a_2,b_2c_2)U(b_3,c_3)\\
&=\sigma^{-1}(a_1,b_1,c_1)g(a_2,b_2c_2)q^R(s(b_3c_3),s(a_3)))g(b_4,c_4)q^R(s(c_5),s(b_5))\\
&\stackrel{(\ref{1h})}{=}\sigma^{-1}(a_1,b_1,c_1)g(a_2,b_2c_2)g(b_3,c_3)q^R(s(c_4)s(b_4),s(a_3)))q^R(s(c_5),s(b_5))\\
&\stackrel{(\ref{1j})}{=}g(a_1b_1,c_1)g(b_2,c_2)\sigma^{-1}(s(c_2),s(b_3),s(a_3))q^R(s(c_3)s(b_4),s(a_4)))q^R(s(c_4),s(b_5))\\
&\stackrel{(\ref{2c})}{=}g(a_1b_1,c_1)g(b_2,c_2)\sigma(s(c_2)(s(b_3)s(a_3)),a_6,b_6)f(a_5,b_5)q^R(s(c_3),s(b_4)s(a_4))\\
&=\sigma(s((a_1b_1)c_1),a_6,b_6)g(a_2b_2,c_2)g(b_3,c_3)f(a_4,b_4)q^R(s(c_3),s(a_3b_3))\\
&=\sigma(s((a_1b_1)c_1),a_4,b_4)g(a_2b_2,c_2)q^R(s(c_3),s(a_3b_3))\\
&=\sigma(s((a_1b_1)c_1),a_3,c_3)U(a_2b_2,c_2),
\end{align*}
which proves (\ref{2e}). And
\begin{align*}
&p^L(s(a_1b_1),a_3)U(a_2,b_2)\\
&=\sigma(a_1b_1,s(a_3b_3),a_6)\b(a_2b_2)g(a_4,b_4)q^R(s(b_5),s(a_5))\\
&=\sigma(a_1b_1,s(b_4)s(a_4),a_6)\b(a_2b_2)g(a_3,b_3)q^R(s(b_5),s(a_5))\\
&=\sigma(a_1b_1,s(b_4)s(a_4),a_8)\b(a_2b_2)g(a_3,b_3)\sigma(s(b_5),s(a_5),a_7)\a(a_6)\\
&=\sigma(a_1b_1,s(b_3)s(a_3),a_7)\chi(a_2,b_2)\sigma(s(b_4),s(a_4),a_6)\a(a_5)\\
&=\underline{\sigma(a_1b_1,s(b_7)s(a_6),a_{10})\sigma(a_2b_2,s(b_6),s(a_5))\sigma(s(b_8),s(a_7),a_9)}\\
&\quad\sigma^{-1}(a_3,b_3,s(b_5))\b(a_4)\b(b_4)\a(a_8)\\
&=\sigma(a_1b_1,s(b_7),s(a_6)a_{8})\sigma((a_2b_2)s(b_6),s(a_5),a_9)\sigma^{-1}(a_3,b_3,s(b_5))\b(a_4)\b(b_4)\a(a_7)\\
&=\sigma((a_1b_1)s(b_5),s(a_4),a_5)\sigma^{-1}(a_2,b_2,s(b_4))\b(a_3)\b(b_3)\a(a_4)\\
&=\sigma(a_2(b_2s(b_4)),s(a_4),a_6)\sigma^{-1}(a_1,b_1,s(b_5))\b(a_3)\b(b_3)\a(a_5)\\
&=\sigma(a_2,s(a_4),a_6)\sigma^{-1}(a_1,b_1,s(b_3))\b(a_3)\b(b_2)\a(a_5)\\
&=p^R(a,b).
\end{align*}
The proof is completed.
\end{proof}

\begin{definition}
Let $H$ be a dual quasi-Hopf algebra. A $k$-space $M$ is called a left-left Yetter-Drinfeld module if $M$ is a left $H$-comodule (denote the left coaction by $\l_M:M\rightarrow H\o M$, $m\mapsto m_{(-1)}\o m_{(0)}$) and $H$ acts on $M$ from the left (denote the left action by $h\cdot m$) such that for all $m\in M$ and $h\in H$ the following conditions hold:
\begin{eqnarray}
&&(1)~\sigma(h_1,g_1,m_{(-1)})\sigma((h_2g_2\cdot m_{(0)})_{(-1)},h_3,g_3)(h_2g_2\cdot m_{(0)})_{(0)}\nonumber\\
&&\quad =\sigma(h_1,(g_1\cdot m)_{(-1)},g_2)h_2\cdot(g_1\cdot m)_{(0)},\label{2h}\\
&&(2)~1_H\cdot m=m,\label{2i}\\
&&(3)~h_1m_{(-1)}\o h_2\cdot m_{(0)}=(h_1\cdot m)_{(-1)}h_2\o(h_1\cdot m)_{(0)},\label{2j}
\end{eqnarray}
for all $h,g\in H$ and $m\in M$.
The category of left-left Yetter-Drinfeld modules over $H$ is denoted by $_H^H\mathcal{YD}$ with the morphisms being left $H$-linear and left $H$-colinear.
\end{definition}

\begin{proposition}
Let $H$ be a dual quasi-Hopf algebra, $M\in\mathcal{^HM}$, and $\cdot:H\o M\rightarrow M$ a $k$-linear map satisfying (\ref{2h})--(\ref{2i}). Then (\ref{2j}) is equivalent to
\begin{align}
&(h\cdot m)_{(-1)}\o(h\cdot m)_{(0)}\nonumber\\
&=q^R((h_1m_{(-1)})_1,s(h_5))(h_1m_{(-1)})_2s(h_4)
\o p^R((h_2\cdot m_{(0)})_{(-1)},h_3)(h_2\cdot m_{(0)})_{(0)},\label{2n}
\end{align}
for all $h\in H,m\in M$.
\end{proposition}

\begin{proof}
The proof is similar to that of \cite{Ba2}.
%
%
\end{proof}

\begin{proposition}
Let $H$ be a dual quasi-Hopf algebra. Then $H$ is an object in the category $_H^H{\mathcal{YD}}$ with the structures:
\begin{align}
&\l(h)=h_1s(h_3)\o h_2,\\
&h\triangleright h'=\sigma(h_1,h'_1,s(h'_7))\sigma(h_2h'_2,s(h'_5)s(h_5),h_7)g(h_4,h'_4)q^R(s(h'_6),s(h_6))h_3h'_3,
\end{align}
for all $h,h'\in H$.
\end{proposition}

\begin{proof}
It is easy to see that $H$ is a left $H$-comodule via $\l$. By the element $U$ defined in Lemma 2.2, we obtain that
\begin{equation}
h\triangleright h'=\sigma(h_1,h'_1,s(h'_6))\sigma(h_2h'_2,s(h_4h'_4),h_6)U(h_5,h'_5)h_3h'_3.
\end{equation}
For all $a,b,h\in H$, we have
\begin{align*}
&\sigma(a_1,b_1,h_1s(h_3))\sigma((a_2b_2\triangleright h_2)_{(-1)},a_3,b_3)(a_2b_2\triangleright h_2)_{(0)}\\
&=\sigma(a_1,b_1,h_1s(h_8))\sigma(((a_4b_4)h_4)_{(-1)},a_8,b_8)\sigma(a_2b_2,h_2,s(h_7))\\
&\quad\sigma((a_3b_3)h_3,s((a_5b_5)h_5),a_7b_7)U(a_6b_6,h_6)((a_4b_4)h_4)_{(0)}\\
&=\sigma(a_1,b_1,h_1s(h_{10}))\sigma(a_2b_2,h_2,s(h_9))\\
&\quad\sigma((a_3b_3)h_3,s((a_7b_7)h_7),a_9b_9)\sigma(((a_4b_4)h_4)s((a_6b_6)h_6),a_{10},b_{10})U(a_8b_8,h_8)(a_5b_5)h_5\\
&\stackrel{(\ref{1c})}{=}\sigma(a_1,b_1,h_1s(h_{11}))\sigma(a_2b_2,h_2,s(h_{10}))\sigma(s((a_8b_8)h_8),a_{10},b_{10})\\
&\quad\sigma((a_3b_3)h_3,s((a_7b_7)h_7)a_{11},b_{11})\sigma((a_4b_4)h_4,s((a_6b_6)h_6),a_{12})U(a_9b_9,h_9)(a_5b_5)h_5\\
&\stackrel{(\ref{1a})}{=}\sigma(a_1,b_1,h_1s(h_{13}))\sigma(a_2b_2,h_2,s(h_{12}))\sigma(a_{9},b_{9},h_{9})\sigma(s((a_{10}b_{10})h_{10}),a_{12},b_{12})U(a_{11}b_{11},h_{11})\\
&\quad\sigma^{-1}(a_3,b_3,h_3)\sigma(a_4(b_4h_4),s(a_8(b_8h_8))a_{13},b_{13})\sigma(a_5(b_5h_5),s(a_7(b_7h_7)),a_{14})a_6(b_6h_6)\\
&\stackrel{(\ref{2e})}{=}\sigma(a_1,b_1,h_1s(h_{12}))\sigma(a_2b_2,h_2,s(h_{11}))U(a_{9},b_{9}h_{9})U(b_{10},h_{10})\\
&\quad\sigma^{-1}(a_3,b_3,h_3)\sigma(a_4(b_4h_4),s(a_8(b_8h_8))a_{10},b_{11})\sigma(a_5(b_5h_5),s(a_7(b_7h_7)),a_{11})a_6(b_6h_6)\\
&\stackrel{(\ref{1c})}{=}\sigma(b_1,h_1,s(h_{12}))\sigma(a_1,b_2h_2,s(h_{11}))U(a_{7},b_{8}h_{9})U(b_{9},h_{10})\\
&\quad\sigma(a_2(b_3h_4),s(a_6(b_7h_8))a_{8},b_{10})\sigma(a_3(b_4h_5),s(a_5(b_6h_7)),a_{9})a_4(b_5h_6)\\
&\stackrel{(\ref{2d})}{=}\sigma(b_1,h_1,s(h_{11}))\sigma(a_1,b_2h_2,s(h_{10}))U(a_6,b_7h_7)U(b_{9},h_{9})\\
&\quad\sigma(a_2(b_3h_3),s(b_8h_8),b_{10})\sigma(a_3(b_4h_4),s(a_5(b_6h_6)),a_{7})a_4(b_5h_5)\\
&\stackrel{(\ref{2d})}{=}\sigma(b_{1},h_{1},s(h_{11}))\sigma(a_1,b_{2}h_2,s(h_{10}))\sigma(a_2(b_{3}h_3),s(b_{8}h_{8}),b_{10})U(b_{9},h_{9})\\
&\sigma(a_3(b_{4}h_4),s(a_5(b_{6}h_6)),a_7)U(a_6,b_{7}h_7)a_4(b_{5}h_5)\\
&=\sigma(b_{1},h_{1},s(h_{11}))\sigma(a_1,b_{2}h_2,s(b_{9}h_{9})b_{11})\sigma(a_2(b_{3}h_3),s(b_{8}h_{8}),b_{12})U(b_{10},h_{10})\\
&\sigma(a_3(b_{4}h_4),s(a_5(b_{6}h_6)),a_7)U(a_6,b_{7}h_7)a_4(b_{5}h_5)\\
&\stackrel{(\ref{1c})}{=}\sigma(b_{2}h_2,s(b_{11}h_{11}),b_{13})\sigma(a_1,(b_{3}h_3)s(b_{10}h_{10}),b_{14})\sigma(a_2,b_{4}h_4,s(b_{9}h_{9}))U(b_{12},h_{12})\\
&\sigma(b_{1},h_{1},s(h_{13}))\sigma(a_3(b_{5}h_5),s(a_5(b_{7}h_7)),a_7)U(a_6,b_{8}h_8)a_4(b_{6}h_6)\\
&=\sigma(a_1,(b_1\triangleright h)_{(-1)},b_2)a_2\triangleright (b_1\triangleright h)_{(0)}.
\end{align*}
Thus the relation (\ref{2h}) is satisfied.
And
\begin{align*}
&a_1h_{(-1)}\o a_2\triangleright h_{(0)}=a_1(h_1s(h_3))\o a_2\triangleright h_2\\
   &=a_1(h_1s(h_8))\sigma(a_2,h_2,s(h_7))\sigma(a_3h_3,s(a_5h_5),a_7)U(a_6,h_6)\o a_4h_4\\
   &\stackrel{(\ref{1a})}{=}\sigma(a_1,h_1,s(h_8))(a_2h_2)s(h_7)\sigma(a_3h_3,s(a_5h_5),a_7)U(a_6,h_6)\o a_4h_4\\
   &\stackrel{(\ref{2d})}{=}\sigma(a_1,h_1,s(h_8))\sigma(a_3h_3,s(a_5h_5),a_9)(a_2h_2)[s(a_6h_6)a_8]U(a_7,h_7)\o a_4h_4\\
   &\stackrel{(\ref{1a})}{=}\sigma(a_1,h_1,s(h_8))\sigma(a_2h_2,s(a_6h_6),a_8)[(a_3h_3)s(a_5h_5)]U(a_7,h_7)a_9\o a_4h_4\\
   &=(a_1\triangleright h)_{(-1)}a_2\o (a_1\triangleright h)_{(0)}.
\end{align*}
The relation (\ref{2j}) is satisfied. Obviously $1_H\triangleright h=h$.
Hence $H$ is an object in the category $_H^H{\mathcal{YD}}$.
The proof is completed.
\end{proof}

\begin{example}
Let $(H,\varphi)$ be a coquasitriangular dual quasi-Hopf algebra. Then any left $H$-comodule $M$ is a left Yetter-Drinfeld module over $H$. Indeed for all $g\in H,m\in M$, define
$$g\cdot m=\varphi(m_{(-1)},g)m_{(0)}.$$
Then for the relation (\ref{2h})
\begin{align*}
&\sigma(h_1,g_1,m_{(-1)})\sigma((h_2g_2\cdot m_{(0)})_{(-1)},h_3,g_3)(h_2g_2\cdot m_{(0)})_{(0)}\\
&=\sigma(h_1,g_1,m_{(-1)1})\sigma( m_{(-1)3},h_3,g_3)\varphi(m_{(-1)2},h_2g_2) m_{(0)}\\
&=\sigma(h_1,g_1,m_{(-1)1})\sigma( m_{(-1)7},h_6,g_6)\sigma^{-1}(h_2,g_2,m_{(-1)2})\varphi(m_{(-1)3},g_3)\\
&~~~~\sigma(h_3,m_{(-1)4},g_4)\varphi(m_{(-1)5},h_4)\sigma^{-1}(m_{(-1)6},h_5,g_5)m_{(0)}\\
&=\varphi(m_{(-1)1},g_1)\sigma(h_1,m_{(-1)2},g_2)\varphi(m_{(-1)3},h_2)m_{(0)}\\
&=\sigma(h_1,(g_1\cdot m)_{(-1)},g_2)h_2\cdot(g_1\cdot m)_{(0)}.
\end{align*}
And for the relation (\ref{2j})
\begin{align*}
&h_1m_{(-1)}\o h_2\cdot m_{(0)}=h_1m_{(-1)1}\o m_{(0)}\varphi(m_{(-1)2},h_2)\\
&=\varphi(m_{(-1)1},h_1)m_{(-1)2}h_2\o m_{(0)}=(h_1\cdot m)_{(-1)}h_2\o(h_1\cdot m)_{(0)}.
\end{align*}
\end{example}

\begin{proposition}\cite{Ba2}
Let $H$ be a dual quasi-bialgebra and $\mathcal{^HM}$ the category of left $H$-comodules. Then we have category isomorphism $_H^H\mathcal{YD}\cong \mathcal{W}_r(^H\mathcal{M})$.
\end{proposition}

Using Proposition 2.7, we could give the monoidal structure of $_H^H\mathcal{YD}$ explicitly. The action of $H$ on the tensor product $M\o N$ of two left-left Yetter-Drinfeld modules $M$ and $N$ is given by
\begin{align}
h\cdot(m\o n)=&\sigma(h_1,m_{(-1)},n_{(-1)1})\sigma^{-1}((h_2\cdot m_{(0)})_{(-1)1},h_3,n_{(-1)2})\nonumber\\
&\sigma((h_2\cdot m_{(0)})_{(-1)2},(h_4\cdot n_{(0)})_{(-1)},h_5)(h_2\cdot m_{(0)})_{(0)}\o (h_4\cdot n_{(0)})_{(0)},
\end{align}
for all $m\in M,n\in N$. The braiding is given by
\begin{equation}
c_{_{N,M}}(n\o m)=n_{(-1)}\cdot m\o n_{(0)}.
\end{equation}
Furthermore we have the following result.

\begin{theorem}
The braiding $c$ is invertible with the inverse $c^{-1}_{_{N,M}}: M\o N\rightarrow N\o M$ given by
\begin{align}
c^{-1}_{_{N,M}}(m\o n)=&q^L(s^{-1}(n_{(-1)6}),m_{(-1)1}n_{(-1)1})\sigma(s^{-1}(n_{(-1)5}),m_{(-1)2},n_{(-1)2})\nonumber\\
p^R((s^{-1}&(n_{(-1)4})\cdot m_{(0)})_{(-1)},s^{-1}(n_{(-1)3}))n_{(0)}\o(s^{-1}(n_{(-1)4})\cdot m_{(0)})_{(0)}.
\end{align}
\end{theorem}

\begin{proof}
For all $m\in M,n\in N$,
\begin{align*}
&c^{-1}_{_{N,M}}(c_{_{N,M}}(n\o m))=c^{-1}_{_{N,M}}(n_{(-1)}\cdot m\o n_{(0)})\\
&=q^L(s^{-1}(n_{(-1)7}),(n_{(-1)1}\cdot m)_{(-1)1}n_{(-1)2})\sigma(s^{-1}(n_{(-1)6}),(n_{(-1)1}\cdot m)_{(-1)2},n_{(-1)3})\\
&\quad p^R((s^{-1}(n_{(-1)5})\cdot (n_{(-1)1}\cdot m)_{(0)})_{(-1)},s^{-1}(n_{(-1)4}))n_{(0)}\o(s^{-1}(n_{(-1)5})\cdot (n_{(-1)1}\cdot m)_{(0)})_{(0)}\\
&=q^L(s^{-1}(n_{(-1)7}),(n_{(-1)1}\cdot m)_{(-1)}n_{(-1)2})\sigma(s^{-1}(n_{(-1)6}),(n_{(-1)1}\cdot m)_{(0)(-1)},n_{(-1)3})\\
&\quad p^R((s^{-1}(n_{(-1)5})\cdot (n_{(-1)1}\cdot m)_{(0)(0)})_{(-1)},s^{-1}(n_{(-1)4}))n_{(0)}\o(s^{-1}(n_{(-1)5})\cdot (n_{(-1)1}\cdot m)_{(0)(0)})_{(0)}\\
&\stackrel{(\ref{2j})}{=}q^L(s^{-1}(n_{(-1)7}),n_{(-1)1}m_{(-1)})\sigma(s^{-1}(n_{(-1)6}),(n_{(-1)2}\cdot m_{(0)})_{(-1)},n_{(-1)3})\\
&\quad p^R((s^{-1}(n_{(-1)5})\cdot (n_{(-1)2}\cdot m_{(0)})_{(0)})_{(-1)},s^{-1}(n_{(-1)4}))n_{(0)}\o(s^{-1}(n_{(-1)5})\cdot (n_{(-1)2}\cdot m_{(0)})_{(0)})_{(0)}\\
&\stackrel{(\ref{2h})}{=}q^L(s^{-1}(n_{(-1)9}),n_{(-1)1}m_{(-1)1})\\
&\quad\sigma(s^{-1}(n_{(-1)8}),n_{(-1)2},m_{(-1)2})\sigma((s^{-1}(n_{(-1)7})n_{(-1)3}\cdot m_{(0)})_{(-1)}),s^{-1}(n_{(-1)6}),n_{(-1)4})\\
&\quad p^R((s^{-1}(n_{(-1)7})n_{(-1)3}\cdot m_{(0)})_{(0)(-1)},s^{-1}(n_{(-1)5}))n_{(0)}\o(s^{-1}(n_{(-1)7})n_{(-1)3}\cdot m_{(0)})_{(0)(0)}\\
&=q^L(s^{-1}(n_{(-1)7}),n_{(-1)1}m_{(-1)1})\sigma(s^{-1}(n_{(-1)6}),n_{(-1)2},m_{(-1)2})\b(s^{-1}(n_{(-1)4}))\\
&\quad n_{(0)}\o s^{-1}(n_{(-1)5})n_{(-1)3}\cdot m_{(0)}\\
&\stackrel{(\ref{1e})}{=}q^L(s^{-1}(n_{(-1)5}),n_{(-1)1}m_{(-1)1})\sigma(s^{-1}(n_{(-1)4}),n_{(-1)2},m_{(-1)2})\b(s^{-1}(n_{(-1)3}))\\
&\quad n_{(0)}\o  m_{(0)}\\
&=q^L(s^{-1}(n_{(-1)3}),n_{(-1)1}m_{(-1)1})p^L(n_{(-1)2},m_{(-1)2})n_{(0)}\o  m_{(0)}\\
&\stackrel{(\ref{2g})}{=}n\o m.
\end{align*}
That is, $c^{-1}_{_{N,M}}\circ c_{_{N,M}}=id_{N\o M}$. Similarly $c_{_{N,M}}\circ c^{-1}_{_{N,M}}=id_{M\o N}$.
The proof is completed.
\end{proof}

We also introduce left-right, and right-left Yetter-Drinfeld modules in the following definition.

\begin{definition}
Let $H$ be a dual quasi-Hopf algebra.
\begin{itemize}
  \item [(a)] A right-left Yetter-Drinfeld module is a left $H$-comodule $M$ together with a right $H$-action $\cdot$ on $M$ such that for all $g,h\in H,m\in M$,
      \begin{eqnarray}
      &&(1)~\sigma^{-1}(m_{(-1)},h_1,g_1)\sigma^{-1}(h_3,g_3,(m_{(0)}\cdot h_2g_2)_{(-1)})(m_{(0)}\cdot h_2g_2)_{(0)}\nonumber\\
      &&=\sigma^{-1}(h_2,(m\cdot h_1)_{(-1)},g_1)(m\cdot h_1)_{(0)}\cdot g_2,\\
      &&(2)~ m\cdot 1=m,\\
      &&(3)~ h_2(m\cdot h_1)_{(-1)}\o (m\cdot h_1)_{(0)}=m_{(-1)}h_1\o m_{(0)}\cdot h_2.
      \end{eqnarray}
The category of right-left Yetter-Drinfeld modules over $H$ is denoted by $^H\mathcal{YD}_H$.
  \item [(b)] A left-right Yetter-Drinfeld module is a right $H$-comodule $M$ together with a left $H$-action $\cdot$ on $M$ such that for all $g,h\in H,m\in M$,
      \begin{eqnarray}
      &&(1)~\sigma^{-1}((h_2g_2\cdot m_{(0)})_{(1)},h_1,g_1)\sigma^{-1}(h_3,g_3,m_{(1)})(h_2g_2\cdot m_{(0)})_{(0)}\nonumber\\
      &&=\sigma^{-1}(h_2,(g_2\cdot m)_{(1)},g_1)h_1\cdot(g_2\cdot m)_{(0)},\\
      &&(2)~ 1\cdot m=m,\\
      &&(3)~ (h_2\cdot m)_{(0)}\o(h_2\cdot m)_{(1)}h_1=h_1\cdot m_{(0)}\o h_2m_{(1)}.
       \end{eqnarray}
The category of left-right Yetter-Drinfeld modules over $H$ is denoted by $_H\mathcal{YD}^H$.
\end{itemize}
\end{definition}

\begin{theorem}
Let $H$ be a dual quasi-bialgebra. Then we have the following category isomorphisms:
$$\mathcal{W}_l(^H\mathcal{M})\cong\! ^H\mathcal{YD}_H,\quad \mathcal{W}_r(\mathcal{M}^H)\cong\! _H\mathcal{YD}^H.$$
If $H$ is a dual quasi-Hopf algebra, then these three weak centers are equal to the centers.
\end{theorem}

\begin{proof}
The proof is straightforward and left to the reader.
\end{proof}

$\bullet$ The prebraided monoidal structure on $\mathcal{W}_l(^H\mathcal{M})$ induces a monoidal structure on $^H\mathcal{YD}_H$. We find that the action on $M\o N$ of two right-left Yetter-Drinfeld modules $M$ and $N$ is given by
\begin{align*}
(m\o n)\cdot h=&\sigma^{-1}(m_{(-1)1},n_{(-1)},h_1)\sigma(m_{(-1)2},h_3,(n_{(0)}\cdot h_2)_{(-1)1})\nonumber\\
               &\sigma^{-1}(h_5,(m_{(0)}\cdot h_4)_{(-1)},(n_{(0)}\cdot h_2)_{(-1)2})(m_{(0)}\cdot h_4)_{(0)}\o(n_{(0)}\cdot h_2)_{(0)},
\end{align*}
for all $h\in H,m\in M$, and $n\in N$.

The braiding $d_{_{M,N}}:M\o N\rightarrow N\o M$ is given by
\begin{align*}
d_{_{M,N}}(m\o n)=n_{(0)}\o m\cdot n_{(-1)}.
\end{align*}
In the case when $H$ is a dual quasi-Hopf algebra, the inverse of $d_{_{M,N}}$ is given by
\begin{align*}
d^{-1}_{_{M,N}}(n\o m)=&q^R(n_{(-1)1}m_{(-1)1},s(n_{(-1)6}))\sigma^{-1}(n_{(-1)2},m_{(-1)2},s(n_{(-1)5}))\nonumber\\
              &p^L(s(n_{(-1)3}),(m_{(0)}\cdot s(n_{(-1)4}))_{(-1)})(m_{(0)}\cdot s(n_{(-1)4}))_{(0)}\o n_{(0)}.
\end{align*}

$\bullet$ The prebraided monoidal structure on $_H\mathcal{YD}^H$: for $M,N\in\! _H\mathcal{YD}^H$, the action on $M\o N$ is given by
\begin{align*}
h\cdot(m\o n)=&\sigma^{-1}(h_5,m_{(1)},n_{(1)2})\sigma((h_4\cdot m_{(0)})_{(1)2},h_3,n_{(1)1})\nonumber\\
              &\sigma^{-1}((h_4\cdot m_{(0)})_{(1)1},(h_2\cdot n_{(0)})_{(1)},h_1)(h_4\cdot m_{(0)})_{(0)}\o(h_2\cdot n_{(0)})_{(0)},
\end{align*}
and
$$(m\o n)_{(0)}\o (m\o n)_{(1)}=m_{(0)}\o n_{(0)}\o m_{(1)}n_{(1)},$$
the braiding is the following:
$$t_{_{M,N}}(m\o n)=m_{(1)}\cdot n\o m_{(0)}.$$
In the case when $H$ is a dual quasi-Hopf algebra, the inverse of $t_{_{M,N}}$ is given by
\begin{align*}
t^{-1}_{_{M,N}}(n\o m)&=p^L(s(m_{(1)1}),n_{(1)2}m_{(1)6})\sigma^{-1}(s(m_{(1)2}),n_{(1)1},m_{(1)5})\\
&q^R((s(m_{(1)3})\cdot n_{(0)})_{(1)},s(m_{(1)4}))m_{(0)}\o(s(m_{(1)3})\cdot n_{(0)})_{(0)}.
\end{align*}

\begin{proposition}
We have an isomorphism of monoidal categories
$$F:\overline{^{H^{op,cop}}\mathcal{M}}\rightarrow \mathcal{M}^H,$$
where $F$ acts on objects and morphisms as identity, and the right $H$-coaction is given by $m_{[0]}\o m_{[1]}=m_{(-1)}\o m_{(0)}$. Similarly we have
$$\overline{\mathcal{M}^{H^{op,cop}}}\rightarrow\! ^H\mathcal{M}.$$
\end{proposition}

\begin{proof}
We only need to verify that $F$ preserves the monoidal structure. For all objects $M, N\in\! ^{H^{op,cop}}\mathcal{M}$,
$$(n\o m)_{(-1)}\o (n\o m)_{(0)}=m_{(-1)}n_{(-1)}\o n_{(0)}\o m_{(0)}.$$
The associativity constraint $a_{P,N,M}:(P\o N)\o M\rightarrow P\o (N\o M)$ is defined as
$$a_{P,N,M}(p\o n\o m)=\sigma^{-1}(m_{(-1)},n_{(-1)},p_{(-1)})p_{(0)}\o n_{(0)}\o m_{(0)}.$$
As for the monoidal structure on $\overline{^{H^{op,cop}}\mathcal{M}}$, we have $M\overline{\o}N=N\o M$. Then
$$(m\overline{\o}n)_{(-1)}\o (m\overline{\o}n)_{(0)}=m_{(-1)}n_{(-1)}\o (m_{(0)}\overline{\o} n_{(0)}).$$
The associativity constraint $\overline{a}_{M,N,P}:(M\overline{\o}N)\overline{\o} P\rightarrow M\overline{\o}(N\overline{\o} P)$ is defined as
$$\overline{a}_{M,N,P}(m\overline{\o} n\overline{\o} p)=\sigma(m_{(-1)},n_{(-1)},p_{(-1)})m_{(0)}\overline{\o} n_{(0)}\overline{\o} p_{(0)}.$$
The proof is completed.
\end{proof}

\begin{proposition}
Let $H$ be a dual quasi-Hopf algebra. Then we have the following isomorphisms of braided monoidal categories:
$$\mathcal{YD}^H_H\cong\! \overline{ ^{H^{op,cop}} _{H^{op,cop}}\mathcal{YD}},\ \ ^H\mathcal{YD}_H\cong\! \overline{_{H^{op,cop}}\mathcal{YD}^{H^{op,cop}}}.$$
\end{proposition}

\begin{proof}
By Proposition 1.1 and Proposition 2.12, we obtain
$$\mathcal{YD}^H_H\cong\! \mathcal{Z}_l(\mathcal{M}^H)\cong\! \mathcal{Z}_l(\overline{^{H^{op,cop}}\mathcal{M}})\cong\! \overline{\mathcal{Z}_r(^{H^{op,cop}}\mathcal{M})}\cong\! \overline{^{H^{op,cop}}_{H^{op,cop}}\mathcal{YD}},$$
and
$$^H\mathcal{YD}_H\cong\! \mathcal{Z}_l(^H\mathcal{M})\cong\! \mathcal{Z}_l(\overline{\mathcal{M}^{H^{op,cop}}})\cong\! \overline{\mathcal{Z}_r(\mathcal{M}^{H^{op,cop}})}\cong\! \overline{_{H^{op,cop}}\mathcal{YD}^{H^{op,cop}}}.$$
\end{proof}

\begin{proposition}\cite{BCP}
Let $\mathcal{C}$ be a monoidal category. Then we have a braided isomorphism of braided monoidal categories $T:\mathcal{Z}_l\mathcal{(C)}\rightarrow\mathcal{Z}_r\mathcal{(C)}^{in}$, given by
$$T(V,s_{_{V,-}})=(V,s^{-1}_{_{V,-}})\quad \hbox{and}\quad T(\upsilon)=\upsilon.$$
\end{proposition}

Of course, the conclusion holds for the right center. By this isomorphism, we have the following result.

\begin{proposition}
Let $H$ be a dual quasi-Hopf algebra, and $^H\mathcal{YD}^{in}_H$ the category $^H\mathcal{YD}_H$ with the braiding
$$\tilde{c}_{_{M,N}}=c^{-1}_{_{M,N}}.$$
Then we have an isomorphism of braided monoidal categories
$$T:\ ^H\mathcal{YD}^{in}_H\cong\! ^H_H\mathcal{YD},$$
defined as follows. For $M\in\! ^H\mathcal{YD}_H$, $T(M)=M$ as a left $H$-comodule; the left $H$-action is given by
\begin{align*}
h\triangleright m&=q^R(h_{1}m_{(-1)1},s(h_{6}))\sigma^{-1}(h_{2},m_{(-1)2},s(h_{5}))\\
              &p^L(s(h_{3}),(m_{(0)}\cdot s(h_{4}))_{(-1)})(m_{(0)}\cdot s(h_{4}))_{(0)},
\end{align*}
for all $h\in H,m\in M$, where $\cdot$ is the right action of $H$ on $M$. The functor $T$ sends a morphism to itself.
\end{proposition}

\begin{proof}
The functor $T$ is just the composition of the isomorphisms
$$^H\mathcal{YD}^{in}_H\rightarrow\mathcal{Z}_l(^H\mathcal{M})^{in}\rightarrow\mathcal{Z}_r(^H\mathcal{M})\rightarrow\!  ^H_H\mathcal{YD}.$$
For $M\in\! ^H\mathcal{YD}^{in}_H$, we compute the corresponding left-left Yetter-Drinfeld module structure on $M$ is the following:
\begin{align*}
&h\triangleright m=(id\o\varepsilon)s^{-1}_{_{M,H}}(h\o m)=(id\o\varepsilon)\tilde{c}_{_{M,H}}(h\o m)=(id\o\varepsilon)d^{-1}_{_{M,H}}(h\o m)\\
        &=q^R(h_{1}m_{(-1)1},s(h_{6}))\sigma^{-1}(h_{2},m_{(-1)2},s(h_{5}))
              p^L(s(h_{3}),(m_{(0)}\cdot s(h_{4}))_{(-1)})(m_{(0)}\cdot s(h_{4}))_{(0)},
\end{align*}
as claimed.
\end{proof}

In the same way, we have the following result.

\begin{proposition}
Let $H$ be a dual quasi-Hopf algebra. Then the categories $\mathcal{YD}^H_H$ and $_H\mathcal{YD}^{Hin}$ are isomorphic as braided monoidal categories.
\end{proposition}

\section{The rigid braided category $_H^H\mathcal{YD}^{fd}$}
\def\theequation{3.\arabic{equation}}
\setcounter{equation} {0} \hskip\parindent

It is well known that the category of finite dimensional Yetter-Drinfeld modules over a Hopf algebra with bijective antipode is rigid. Since $^H\mathcal{M}_{fd}$ is rigid, the same result holds for the category of finite dimensional Yetter-Drinfeld modules over a dual quasi-Hopf algebra. In this section we will give the explicit forms.

\begin{proposition}\cite{BCP}
Let $\mathcal{C}$ be a rigid monoidal category. Then the weak left (respectively right) center of $\mathcal{C}$ is a rigid braided monoidal category.
\end{proposition}

For example, for any object $(V,c_{_{-,V}})\in\mathcal{Z}_r\mathcal{(C)}$, $^*(V,c_{_{-,V}})=(^*V,c_{_{-,^*V}})$, with $c_{_{-,^*V}}$ given by the following composition:
\begin{align}
c_{_{X,^*V}}:X\o\! ^*V&\xrightarrow{coev'_V\o (X\o^*V)}(^*V\o V)\o(X\o\! ^*V)\nonumber\\
                     &\xrightarrow{a_{_{^*V,V,(X\o^*V)}}}\ ^*V\o (V\o(X\o\! ^*V))\nonumber\\
                     &\xrightarrow{^*V\o a^{-1}_{_{V,X,^*V}}}\ ^*V\o ((V\o X)\o\! ^*V)\nonumber\\
                     &\xrightarrow{^*V\o c^{-1}_{_{V,X}}\o^*V}\ ^*V\o ((X\o V)\o\! ^*V)\label{3a}\\
                     &\xrightarrow{^*V\o a_{_{V,X,^*V}}}\ ^*V\o (X\o (V\o\! ^*V))\nonumber\\
                     &\xrightarrow{a^{-1}_{_{ ^*V,V,(V\o^*V)}}}\ (^*V\o X)\o (V\o\! ^*V)\nonumber\\
                     &\xrightarrow{^*V\o V\o ev'_v}\ ^*V\o X.\nonumber
\end{align}

\begin{lemma}
Let $H$ be a dual quasi-Hopf algebra. Then for all $a,b,c\in H$, the following relations hold:
\begin{eqnarray}
&&q^L(a_1,b_1c_1)\sigma(a_2,b_2,c_2)=q^L(a_2,b_1)\sigma^{-1}(s(a_1),a_3b_2,c),\label{3b}\\
&&p^R(s(a_1),a_3b_3)q^L(a_2,b_2)q^L(b_1,s(a_4b_4))=f(a,b),\label{3c}
\end{eqnarray}
\end{lemma}

\begin{proof}
By the definition of $q^L$, it is easy to verify (\ref{3b}). Then for all $a,b\in H$,
\begin{align*}
&p^R(s(a_1),a_3b_3)q^L(a_2,b_2)q^L(b_1,s(a_4b_4))\\
&=\sigma^{-1}(s(a_1),a_3b_3,s(a_5b_5))\b(a_4b_4)q^L(a_2,b_2)q^L(b_1,s(a_6b_6))\\
&\stackrel{(\ref{1h},\ref{1i})}{=}\sigma^{-1}(s(a_1),a_3b_3,s(b_5)s(a_5))q^L(a_2,b_2)\chi(a_4,b_4)q^L(b_1,s(b_6)s(a_6))f(a_7,b_7)\\
&=\sigma^{-1}(s(a_1),a_5b_3,s(b_9)s(a_{10}))\sigma^{-1}(s(a_2),a_4,b_2)\a(a_3)\\
&\sigma(a_6b_4,s(b_8),s(a_9))\sigma^{-1}(a_7,b_5,s(b_7))\b(a_8)\b(b_6)q^L(b_1,s(b_{10})s(a_{11}))f(a_{12},b_{11})\\
&\stackrel{(\ref{1c},\ref{1e})}{=}\sigma^{-1}(s(a_2),a_4,b_2)\sigma^{-1}(s(a_1),a_5b_3,s(b_9)s(a_{10}))\sigma^{-1}(a_6,b_4,s(b_8)s(a_9))\\
&\a(a_3)\sigma(b_5,s(b_7),s(a_8))\b(b_6)\b(a_7)q^L(b_1,s(b_{10})s(a_{11}))f(a_{12},b_{11})\\
&\stackrel{(\ref{1c},\ref{1e})}{=}\sigma^{-1}(s(a_1),a_3,b_2(s(b_6)s(a_6)))\a(a_2)\sigma(b_3,s(b_5),s(a_5))\\
&\b(b_4)\b(a_4)q^L(b_1,s(b_{7})s(a_{7}))f(a_{8},b_{8})\\
&\stackrel{(\ref{1a},\ref{1e})}{=}\sigma^{-1}(s(a_1),a_3,s(a_5))\a(a_2)\sigma(b_2,s(b_4),s(a_6))\b(b_3)\b(a_4)q^L(b_1,s(b_{5})s(a_{7}))f(a_{8},b_{6})\\
&\stackrel{(\ref{1a},\ref{1e})}{=}\sigma(b_2,s(b_4),s(a_1))\b(b_3)q^L(b_1,s(b_{5})s(a_{2}))f(a_{3},b_{6})\\
&=p^L(s(b_2),s(a_1))q^L(b_1,s(b_{3})s(a_{2}))f(a_{3},b_{4})\\
&\stackrel{(\ref{2g})}{=}f(a,b),
\end{align*}
as needed. The proof is completed.
\end{proof}

\begin{theorem}
Let $H$ be a dual quasi-Hopf algebra. Then $^H_H\mathcal{YD(H)}^{fd}$ is a braided monoidal rigid category. For a finite dimensional left-left Yetter-Drinfeld module $M$ with basis $\{m_i\}_i$ and dual basis $\{m^i\}_i$, the left and right duals $M^*$ and $^*M$ are equal to Hom$(M,k)$ as a vector space, with the following $H$-action and $H$-coaction:
\begin{itemize}
  \item [(1)] For $^*M$,
  \begin{align}
  \l_{_{^*M}}(\varphi)=&\langle \varphi,m_{i(0)}\rangle s(m_{i(-1)})\o m^i,\label{3d}\\
  h\cdot\varphi=&f(s^{-1}(h_3),m_{i(-1)})g((s^{-1}(h_2)\cdot m_{i(0)})_{(-1)},s^{-1}(h_1))\nonumber\\
                 &\varphi((s^{-1}(h_2)\cdot m_{i(0)})_{(0)})m^i.
  \end{align}
  \item [(2)] For $M^*$,
  \begin{align}
  \l_{_{M^*}}(\varphi')=&\langle\varphi',m_{i(0)}\rangle s^{-1}(m_{i(-1)})\o m^i,\label{3e}\\
  h\cdot\varphi'=&f(s^{-1}(m_{i(-1)}),h_3)g(h_1,s^{-1}((s(h_2)\cdot m_{i(0)})_{(-1)}))\nonumber\\
                 &\varphi'((s(h_2)\cdot m_{i(0)})_{(0)}))m^i,
  \end{align}
\end{itemize}
for all $h\in H,\varphi\in\! ^*M,\varphi'\in M^*$.
\end{theorem}

\begin{proof}
The left $H$-coaction on $^*M$ viewed as an object in $^H_H\mathcal{YD}$ is the same as the left $H$-coaction on $^*M$ viewed as an object in $^H\mathcal{M}$. Now we compute the left $H$-action using (\ref{3a}). For all $h\in H,\varphi\in\! ^*M$,
\begin{align*}
&h\cdot\varphi=(id\o\varepsilon)c_{_{H,^*M}}(h\o\varphi)\\
&=\sigma^{-1}(m^i_{(-1)1},m_{i(-1)2},h_1s((s^{-1}(h_6)\cdot m_{i(0)})_{(-1)8})\sigma(m_{i(-1)3},h_2,s((s^{-1}(h_6)\cdot m_{i(0)})_{(-1)7})\\
&\a(m_{i(-1)1})q^L(s^{-1}(h_{8}),m_{i(-1)4}h_3)\sigma(s^{-1}(h_7),m_{i(-1)5},h_4)p^R((s^{-1}(h_6)\cdot m_{i(0)})_{(-1)1},s^{-1}(h_5))\\
&\sigma^{-1}(h_9,(s^{-1}(h_6)\cdot m_{i(0)})_{(-1)2},s((s^{-1}(h_6)\cdot m_{i(0)})_{(-1)6}))\b((s^{-1}(h_6)\cdot m_{i(0)})_{(-1)4})\\
&\sigma(m^i_{(-1)2},h_{10},(s^{-1}(h_6)\cdot m_{i(0)})_{(-1)3}s((s^{-1}(h_6)\cdot m_{i(0)})_{(-1)5}))\varphi((s^{-1}(h_6)\cdot m_{i(0)})_{(0)})m^i_{(0)}\\
&\stackrel{(\ref{1e})(\ref{3d})}{=}q^L(m_{i(-1)1},h_1s((s^{-1}(h_6)\cdot m_{i(0)})_{(-1)4}))\sigma(m_{i(-1)2},h_2,s((s^{-1}(h_6)\cdot m_{i(0)})_{(-1)3}))\\
&q^L(s^{-1}(h_{8}),m_{i(-1)3}h_3)\sigma(s^{-1}(h_7),m_{i(-1)4},h_4)p^R((s^{-1}(h_6)\cdot m_{i(0)})_{(-1)1},s^{-1}(h_5))\\
&p^R(h_9,(s^{-1}(h_6)\cdot m_{i(0)})_{(-1)2})\varphi((s^{-1}(h_6)\cdot m_{i(0)})_{(0)})m^i\\
&\stackrel{(\ref{3b})}{=}q^L(m_{i(-1)1},h_1s((s^{-1}(h_5)\cdot m_{i(0)})_{(-1)4}))\sigma(m_{i(-1)2},h_2,s((s^{-1}(h_5)\cdot m_{i(0)})_{(-1)3}))\\
&q^L(s^{-1}(h_7),m_{i(-1)3})\sigma^{-1}(h_8,s^{-1}(h_6)m_{i(-1)4},h_3)p^R((s^{-1}(h_5)\cdot m_{i(0)})_{(-1)1},s^{-1}(h_4))\\
&p^R(h_{9},(s^{-1}(h_5)\cdot m_{i(0)})_{(-1)2})
\varphi((s^{-1}(h_5)\cdot m_{i(0)})_{(0)})m^i\\
&\stackrel{(\ref{2i})}{=}q^L(m_{i(-1)1},h_1s((s^{-1}(h_6)\cdot m_{i(0)})_{(-1)5}))\sigma(m_{i(-1)2},h_2,s((s^{-1}(h_6)\cdot m_{i(0)})_{(-1)4}))\\
&q^L(s^{-1}(h_7),m_{i(-1)3})\sigma^{-1}(h_8,(s^{-1}(h_6)\cdot m_{i(0)})_{(-1)1}s^{-1}(h_5),h_3)p^R((s^{-1}(h_6)\cdot m_{i(0)})_{(-1)2},s^{-1}(h_4))\\
&p^R(h_{9},(s^{-1}(h_6)\cdot m_{i(0)})_{(-1)3})
\varphi((s^{-1}(h_6)\cdot m_{i(0)})_{(0)})m^i\\
&\stackrel{(\ref{3b})}{=}q^L(m_{i(-1)1},h_1s((s^{-1}(h_5)\cdot m_{i(0)})_{(-1)5}))\sigma(m_{i(-1)2},h_2,s((s^{-1}(h_5)\cdot m_{i(0)})_{(-1)4}))\\
&q^L(s^{-1}(h_6),m_{i(-1)3})\sigma(h_7,(s^{-1}(h_5)\cdot m_{i(0)})_{(-1)1},s^{-1}(h_4))p^R(h_8(s^{-1}(h_5)\cdot m_{i(0)})_{(-1)2},s^{-1}(h_3))\\
&p^R(h_{9},(s^{-1}(h_5)\cdot m_{i(0)})_{(-1)3})\varphi((s^{-1}(h_5)\cdot m_{i(0)})_{(0)})m^i\\
&\stackrel{(\ref{2s})}{=}q^L(m_{i(-1)1},h_1s((s^{-1}(h_7)\cdot m_{i(0)})_{(-1)6}))\sigma(m_{i(-1)2},h_2,s((s^{-1}(h_7)\cdot m_{i(0)})_{(-1)5}))\\
&q^L(s^{-1}(h_8),m_{i(-1)3})\sigma^{-1}(h_9((s^{-1}(h_7)\cdot m_{i(0)})_{(-1)1}s^{-1}(h_6)),h_3,s((s^{-1}(h_7)\cdot m_{i(0)})_{(-1)4}))\\
&p^R(h_{10},(s^{-1}(h_7)\cdot m_{i(0)})_{(-1)2}s^{-1}(h_5))g((s^{-1}(h_7)\cdot m_{i(0)})_{(-1)3},s^{-1}(h_4))\varphi((s^{-1}(h_7)\cdot m_{i(0)})_{(0)})m^i\\
&\stackrel{(\ref{2j})}{=}q^L(m_{i(-1)1},h_1s((s^{-1}(h_5)\cdot m_{i(0)})_{(-1)4}))\sigma(m_{i(-1)2},h_2,s((s^{-1}(h_5)\cdot m_{i(0)})_{(-1)3}))\\
&q^L(s^{-1}(h_8),m_{i(-1)3})\sigma^{-1}(h_9(s^{-1}(h_7)m_{i(-1)4}),h_3,s((s^{-1}(h_5)\cdot m_{i(0)})_{(-1)2}))\\
&p^R(h_{10},(s^{-1}(h_6)m_{i(-1)5}))g((s^{-1}(h_5)\cdot m_{i(0)})_{(-1)1},s^{-1}(h_4))\varphi((s^{-1}(h_5)\cdot m_{i(0)})_{(0)})m^i\\
&\stackrel{(\ref{2b})}{=}q^L(m_{i(-1)1},h_1s((s^{-1}(h_5)\cdot m_{i(0)})_{(-1)4}))\sigma(m_{i(-1)2},h_2,s((s^{-1}(h_5)\cdot m_{i(0)})_{(-1)3}))\\
&q^L(s^{-1}(h_7),m_{i(-1)4})\sigma^{-1}(m_{i(-1)3},h_3,s((s^{-1}(h_5)\cdot m_{i(0)})_{(-1)2}))\\
&p^R(h_{8},(s^{-1}(h_6)m_{i(-1)5}))g((s^{-1}(h_5)\cdot m_{i(0)})_{(-1)1},s^{-1}(h_4))
\varphi((s^{-1}(h_5)\cdot m_{i(0)})_{(0)})m^i\\
&\stackrel{(\ref{1h})}{=}q^L(m_{i(-1)1},s(s^{-1}(h_3)m_{i(-1)4})g((s^{-1}(h_2)\cdot m_{i(0)})_{(-1)},s^{-1}(h_1))\\
&q^L(s^{-1}(h_5),m_{i(-1)2})p^R(h_{6},s^{-1}(h_4)m_{i(-1)3})\varphi((s^{-1}(h_2)\cdot m_{i(0)})_{(0)})m^i\\
&\stackrel{(\ref{3c})}{=}f(s^{-1}(h_3),m_{i(-1)})g((s^{-1}(h_2)\cdot m_{i(0)})_{(-1)},s^{-1}(h_1))\varphi((s^{-1}(h_2)\cdot m_{i(0)})_{(0)})m^i,
\end{align*}
as claimed. The structure on $M^*$ can be computed in a similar way. The proof is completed.
\end{proof}

\section{The quantum cocommutativity of $H_0$}
\def\theequation{4.\arabic{equation}}
\setcounter{equation} {0} \hskip\parindent

Let $H$ be a dual quasi-Hopf algebra. In Proposition 2.6, we have made $H$ an object in $^H_H\mathcal{YD}$. In this section, we will prove that $H_0$, which was firstly introduced in \cite{BN2}, is a braided cocommutative coalgebra in $^H_H\mathcal{YD}$. First of all, we need the following definition.

\begin{definition}
A coalgebra in $^H_H\mathcal{YD}$ is a Yetter-Drinfeld module $C$ over $H$ satisfying
\begin{itemize}
  \item [(1)] $C$ is a coalgebra in $^H\mathcal{M}$.
  \item [(2)] For all $h\in H,c\in C$,
  \begin{align}
&(h\cdot c)_1\o (h\cdot c)_2\nonumber\\
&=\sigma(h_1,c_{1(-1)},c_{2(-1)1})\sigma^{-1}((h_2\cdot c_{1(0)})_{(-1)1},h_3,c_{2(-1)2})\nonumber\\
&\quad \sigma((h_2\cdot c_{1(0)})_{(-1)2},(h_4\cdot c_{2(0)})_{(-1)},h_5)(h_2\cdot c_{1(0)})_{(0)}\o (h_4\cdot c_{2(0)})_{(0)}.
\end{align}
\end{itemize}
\end{definition}

Let $H$ be a dual quasi-Hopf algebra. Recall from \cite{BN2} that $H$ becomes a left $H$-comodule coalgebra with the following comultiplication
\begin{equation}
h_{\un{1}}\o h_{\un{2}}=\sigma(h_1,s(h_3),h_7s(h_9))\sigma^{-1}(s(h_4),h_6,s(h_{10}))\a(h_5)h_2\o h_8,\label{4a}
\end{equation}
for all $h\in H$. And the counit is $\b$. We denote this structure of $H$ by $H_0$.

By the identities (\ref{1a}) and (\ref{1e}), the comultiplication is equivalent to
\begin{equation}
h_{\un{1}}\o h_{\un{2}}=\sigma(h_1,s(h_5),h_{7})\sigma^{-1}(h_2s(h_4),h_{8},s(h_{10}))\a(h_6)h_3\o h_{9}.\label{4b}
\end{equation}

\begin{proposition}
Let $H$ be a dual quasi-Hopf algebra. Then $H_0$ is a coalgebra in $^H_H\mathcal{YD}$.
\end{proposition}

\begin{proof}
First of all, for all $a,h\in H$,
\begin{align*}
&\sigma(h_1,a_{\un{1}(-1)},a_{\un{2}(-1)1})\sigma^{-1}((h_2\tr a_{\un{1}(0)})_{(-1)1},h_3,a_{\un{2}(-1)2})\\
&\sigma((h_2\tr a_{\un{1}(0)})_{(-1)2},(h_4\tr a_{\un{2}(0)})_{(-1)},h_5)(h_2\tr a_{\un{1}(0)})_{(0)}\o (h_4\tr a_{\un{2}(0)})_{(0)}\\
&=\sigma(h_1,a_{\un{1},1}s(a_{\un{1},3}),a_{\un{2},1}s(a_{\un{2},5}))\sigma^{-1}((h_2\tr a_{\un{1},2})_{1}s((h_2\tr a_{\un{1},2})_{5}), h_3,a_{\un{2},2}s(a_{\un{2},4}))\\
&\sigma((h_2\tr a_{\un{1},2})_{2}s((h_2\tr a_{\un{1},2})_{4}),(h_4\tr a_{\un{2},3})_{1}s((h_4\tr a_{\un{2},3})_{3}),h_5)(h_2\tr a_{\un{1},2})_{3}\o (h_4\tr a_{\un{2},3})_{2}\\
&=\sigma(h_1,a_{\un{1},1}s(a_{\un{1},13}),a_{\un{2},1}s(a_{\un{2},13}))\sigma^{-1}((h_{4}a_{\un{1},4})s(h_{8}a_{\un{1},8}), h_{13},a_{\un{2},2}s(a_{\un{2},12}))\\
&\sigma((h_{5}a_{\un{1},5})s(h_{7}a_{\un{1},7}),(h_{16}a_{\un{2},5})s(h_{18}a_{\un{2},7}),h_{23})(h_{6}a_{\un{1},6})\o (h_{17}a_{\un{2},6})\\
&\sigma(h_2,a_{\un{1},2},s(a_{\un{1},12}))\sigma(h_3a_{\un{1},3},s(a_{\un{1},10})s(h_{10}),h_{12})g(h_9,a_{\un{1},9})q^R(s(a_{\un{1},11}),s(h_{11}))\\
&\sigma(h_{14},a_{\un{2},3},s(a_{\un{2},11}))\sigma(h_{15}a_{\un{2},4},s(a_{\un{2},9})s(h_{20}),h_{22})g(h_{19},a_{\un{2},8})q^R(s(a_{\un{2},10}),s(h_{21}))\\
&=\sigma(h_1,a_{\un{1},1}s(a_{\un{1},13}),a_{\un{2},1}s(a_{\un{2},13}))\underline{\sigma^{-1}((h_{4}a_{\un{1},4})s(h_{8}a_{\un{1},8}), h_{13},a_{\un{2},2}s(a_{\un{2},12}))}\\
&\sigma((h_{5}a_{\un{1},5})s(h_{7}a_{\un{1},7}),(h_{16}a_{\un{2},5})s(h_{18}a_{\un{2},7}),h_{25})(h_{6}a_{\un{1},6})\o (h_{17}a_{\un{2},6})\\
&\sigma(h_2,a_{\un{1},2},s(a_{\un{1},12}))\underline{\sigma(h_3a_{\un{1},3},s(h_{9}a_{\un{1},9}),h_{12})}g(h_{10},a_{\un{1},10})q^R(s(a_{\un{1},11}),s(h_{11}))\\
&\sigma(h_{14},a_{\un{2},3},s(a_{\un{2},11}))\sigma(h_{15}a_{\un{2},4},s(a_{\un{2},9})s(h_{20}),h_{22})g(h_{19},a_{\un{2},8})q^R(s(a_{\un{2},10}),s(h_{21}))\\
&=\sigma(h_1,a_{\un{1},1}s(a_{\un{1},14}),a_{\un{2},1}s(a_{\un{2},17}))\sigma^{-1}(s(h_{9}a_{\un{1},9}), h_{12},a_{\un{2},3}s(a_{\un{2},15}))\\
&\sigma(h_{4}a_{\un{1},4},s(h_{8}a_{\un{1},8})),h_{13}(a_{\un{2},4}s(a_{\un{2},14}))\sigma(h_{14},a_{\un{2},5},s(a_{\un{2},13}))\\
&\sigma((h_{5}a_{\un{1},5})s(h_{7}a_{\un{1},7}),(h_{16}a_{\un{2},7})s(h_{18}a_{\un{2},9}),h_{23})\sigma(h_2,a_{\un{1},2},s(a_{\un{1},13}))\\
&\sigma^{-1}(h_3a_{\un{1},3},s(a_{\un{1},12}),a_{\un{2},2}s(a_{\un{2},16}))g(h_{10},a_{\un{1},10})q^R(s(a_{\un{1},11}),s(h_{11}))\\
&\sigma(h_{15}a_{\un{2},6},s(a_{\un{2},11})s(h_{20}),h_{22})g(h_{19},a_{\un{2},10})q^R(s(a_{\un{2},12}),s(h_{21}))(h_{6}a_{\un{1},6})\o (h_{17}a_{\un{2},8})\\
&=\sigma(h_1,a_{1},s(a_{27}))\sigma^{-1}(s(h_{7}a_{7}), h_{10},a_{13}s(a_{25}))\sigma(h_{2}a_{2},s(h_{6}a_{6}),h_{11}(a_{14}s(a_{24}))\\
&\sigma(h_{12},a_{15},s(a_{23}))\sigma((h_{3}a_{3})s(h_{5}a_{5}),(h_{14}a_{17})s(h_{16}a_{19}),h_{21})g(h_{8},a_{8})q^R(s(a_{9}),s(h_{9}))\\
&\sigma(h_{13}a_{16},s(h_{17}a_{20}),h_{20})g(h_{18},a_{21})q^R(s(a_{22}),s(h_{19}))\sigma^{-1}(s(a_{10}),a_{12},s(a_{26}))\a(a_{11})\\
&h_{4}a_{4}\o h_{15}a_{18}\\
&=\sigma(h_1,a_{1},s(a_{24}))\sigma((h_{3}a_{3})s(h_{5}a_{5}),(h_{15}a_{16})s(h_{17}a_{18}),h_{22})\sigma^{-1}(s(h_{7}a_{7}), h_{12}a_{13},s(a_{23}))\\
&\sigma(h_{2}a_{2},s(h_{6}a_{6}),(h_{13}a_{14})s(a_{22}))\sigma(h_{14}a_{15},s(h_{18}a_{19}),h_{21})g(h_{8},a_{8})q^R(s(a_{9}),s(h_{10}))\\
&\sigma^{-1}(s(a_{10})s(h_9),h_{11},a_{12})g(h_{19},a_{20})q^R(s(a_{21}),s(h_{20}))\a(a_{11})h_{4}a_{4}\o h_{16}a_{17}\\
&=\sigma(h_1,a_{1},s(a_{26}))\sigma(h_{2}a_{2},s(h_{10}a_{10})(h_{15}a_{16}),s(a_{25}))\sigma(h_{3}a_{3},s(h_{9}a_{9}),h_{16}a_{17})\\
&\sigma(((h_{4}a_{4})s(h_{8}a_{8}))(h_{17}a_{18}),s(h_{21}a_{22}),h_{24})\sigma^{-1}((h_{5}a_{5})s(h_{7}a_{7}),h_{18}a_{19},s(h_{20}a_{21}))\\
&\sigma^{-1}(s(h_{11}a_{11}),h_{14},a_{15})g(h_{12},a_{12})g(h_{22},a_{23})q^R(s(a_{24}),s(h_{23}))q^R(s(a_{13}),s(h_{13}))\a(a_{14})\\
&h_6a_6\o h_{19}a_{20}\\
&=\sigma(h_1,a_{1},s(a_{22}))\sigma(h_{2}a_{2},s(h_{8}a_{8}),h_{13}a_{14})\sigma(((h_{3}a_{3})s(h_{7}a_{7}))(h_{14}a_{15}),s(h_{18}a_{19}),h_{21})\\
&\sigma^{-1}((h_{4}a_{4})s(h_{6}a_{6}),h_{15}a_{16},s(h_{17}a_{18}))g(h_{10},a_{10})\sigma^{-1}(s(h_{9}a_{9}),h_{12},a_{13})\\
&g(h_{19},a_{20})q^R(s(a_{11}),s(h_{11}))q^R(s(a_{21}),s(h_{20}))\a(a_{12})h_5a_5\o h_{16}a_{17}\\
&=\sigma(h_1,a_{1},s(a_{20}))\sigma(h_{2}a_{2},s(h_{16}a_{17}),h_{19})\sigma(h_{3}a_{3},s(h_{7}a_{7}),h_{12}a_{13})\\
&\sigma^{-1}((h_{4}a_{4})s(h_{6}a_{6}),h_{13}a_{14},s(h_{15}a_{16}))\underline{g(h_{9},a_{9})}\underline{\sigma^{-1}(s(h_{8}a_{8}),h_{11},a_{12})}\\
&g(h_{17},a_{18})\underline{q^R(s(a_{10}),s(h_{10}))}q^R(s(a_{19}),s(h_{18}))\underline{\a(a_{11})}h_5a_5\o h_{14}a_{15}\\
&=\sigma(h_1,a_{1},s(a_{17}))\sigma(h_{2}a_{2},s(h_{14}a_{14}),h_{17})\sigma(h_{3}a_{3},s(h_{7}a_{7}),h_{10}a_{10})\\
&\sigma^{-1}((h_{4}a_{4})s(h_{6}a_{6}),h_{11}a_{11},s(h_{13}a_{13}))g(h_{8},a_{8})\lambda(h_9,a_9)g(h_{15},a_{15})\\
&q^R(s(a_{16}),s(h_{16}))h_5a_5\o h_{12}a_{12}\\
&=\sigma(h_1,a_{1},s(a_{16}))\sigma(h_{2}a_{2},s(h_{13}a_{13}),h_{16})\sigma(h_{3}a_{3},s(h_{7}a_{7}),h_{9}a_{9})\\
&\sigma^{-1}((h_{4}a_{4})s(h_{6}a_{6}),h_{10}a_{10},s(h_{12}a_{12}))\a(h_{8}a_{8})g(h_{14},a_{14})\\
&q^R(s(a_{15}),s(h_{15}))h_5a_5\o h_{11}a_{11}\\
&=\sigma(h_1,a_1,s(a_7))\sigma(h_2a_2,s(a_5)s(h_5),h_7)g(h_4,a_4)q^R(s(a_6),s(h_6))(h_3a_3)_{\un{1}}\o(h_3a_3)_{\un{2}}\\
&=(h\tr a)_{\un{1}}\o (h\tr a)_{\un{2}}.
\end{align*}
Then
\begin{align*}
\b(h\tr a)&=\sigma(h_1,a_1,s(a_7))\sigma(h_2a_2,s(a_5)s(h_5),h_7)g(h_4,a_4)q^R(s(a_6),s(h_6))\b(h_3a_3)\\
&=\sigma(h_1,a_1,s(a_7))\sigma(h_2a_2,s(h_4 a_4),h_7)g(h_5,a_5)q^R(s(a_6),s(h_6))\b(h_3a_3)\\
&=\sigma(h_1,a_1,s(a_6))\underline{\sigma(h_2a_2,s(h_4 a_4),h_6)}U(h_5,a_5)\underline{\b(h_3a_3)}\\
&=\sigma(h_1,a_1,s(a_4))p^L(s(h_2 a_2),h_4)U(h_3,a_3)\\
&\stackrel{(\ref{2z})}{=}\sigma(h_1,a_1,s(a_3))p^R(h_2, a_2)\\
&=\varepsilon(h)\b(a).
\end{align*}
The proof is completed.
\end{proof}

\begin{lemma}
Let $H$ be a dual quasi-Hopf algebra. We have the following identities
\begin{align}
&(1)~q^R(h_1,s(g_2))\sigma^{-1}(h_2,s(g_1),g_3)=\varepsilon(h)\a(g),\label{4c}\\
&(2)~g(s(h_1),h_3)\a(h_2)=\b(s(h)),\ f(h_1,s(h_3))\b(h_2)=\a(s(h)),\label{4d}\\
&(3)~q^R(s(g_6),s^2(g_2)s(h_2))\sigma(s(g_5),s^2(g_3),s(h_1))f(h_3,s(g_1))\b(s(g_4))=q^L(h,s(g)).\label{4e}
\end{align}
\end{lemma}

\begin{proof}
The verification of (\ref{4c}) is straightforward and left to the reader, and (\ref{4c}) has been introduced in \cite{Ba}. Now we will prove (\ref{4e}). By (\ref{2c}), we obtain
\begin{align}
q^R(h,gk)=&q^R(h_1,g_3)q^R(h_2g_4,k_3)\sigma^{-1}(h_3,g_5,k_4)\nonumber\\
&\sigma^{-1}(h_4(g_6k_5),s^{-1}(k_2),s^{-1}(g_2))g(s^{-1}(k_1),s^{-1}(g_1)).
\end{align}
And using (\ref{1b}) and (\ref{1e})
\begin{align*}
&q^R(h_1g_1,k_1)\sigma^{-1}(h_2,g_2,k_2)\\
&=\sigma(h_1g_1,k_3,s^{-1}(k_1))\sigma^{-1}(h_2,g_2,k_4)\a(s^{-1}(k_2))\\
&=\sigma^{-1}(h_1,g_1,k_5s^{-1}(k_3))\sigma(g_2,k_6,s^{-1}(k_2))\sigma(h_2,g_3k_7,s^{-1}(k_1))\a(s^{-1}(k_4))\\
&=\sigma(g_1,k_4,s^{-1}(k_2))\sigma(h,g_2k_5,s^{-1}(k_1))\a(s^{-1}(k_3))\\
&=q^R(g_1,k_2)\sigma(h,g_2k_3,s^{-1}(k_1)).
\end{align*}
Hence
\begin{align*}
q^R(h,gk)=&q^R(h_1,g_3)q^R(g_4,k_4)\sigma(h_2,g_5k_5,s^{-1}(k_3))\\
&\sigma^{-1}(h_3(g_6k_6),s^{-1}(k_2),s^{-1}(g_2))g(s^{-1}(k_1),s^{-1}(g_1)).
\end{align*}
Now
\begin{align*}
&q^R(s(g_6),s^2(g_2)s(h_2))\sigma(s(g_5),s^2(g_3),s(h_1))f(h_3,s(g_1))\b(s(g_4))\\
&=q^R(s(g_{13}),s^2(g_4))q^R(s^2(g_5),s(h_4))\sigma(s(g_{12}),s^2(g_6)s(h_3),h_5)\\
& \sigma^{-1}(s(g_{11})(s^2(g_7)s(h_2)),h_6,s(g_3))g(h_7,s(g_2))\sigma(s(g_{10}),s^2(g_8),s(h_1))f(h_8,s(g_1))\b(s(g_9))\\
&\stackrel{(\ref{1a})}{=}q^R(s(g_{11}),s^2(g_2))q^R(s^2(g_3),s(h_4))\sigma(s(g_{10}),s^2(g_4)s(h_3),h_5)\\
&\quad \sigma^{-1}((s(g_{8})s^2(g_6))s(h_1),h_6,s(g_1))\sigma(s(g_{9}),s^2(g_5),s(h_2))\b(s(g_7))\\
&\stackrel{(\ref{1e})}{=}q^R(s(g_{9}),s^2(g_2))q^R(s^2(g_3),s(h_4))\underline{\sigma(s(g_{8}),s^2(g_4)s(h_3),h_5)}\sigma^{-1}(s(h_1),h_6,s(g_1))\\
&\quad \underline{\sigma(s(g_{7}),s^2(g_5),s(h_2))}\b(s(g_6))\\
&\stackrel{(\ref{1b})}{=}q^R(s(g_{10}),s^2(g_2))q^R(s^2(g_3),s(h_5))\sigma^{-1}(s^2(g_4),s(h_4),h_6)\sigma(s(g_9),s^2(g_5),s(h_3)h_7)\\
&\quad \sigma(s(g_8)s^2(g_6),s(h_2),h_8)\sigma^{-1}(s(h_1),h_{10},s(g_1))\b(s(g_7))\\
&\stackrel{(\ref{4c})(\ref{1e})}{=}q^R(s(g_{4}),s^2(g_2))\a(h_2)\sigma^{-1}(s(h_1),h_{3},s(g_1))\b(s(g_3))\\
&=\a(h_2)\sigma^{-1}(s(h_1),h_{3},s(g))\\
&=q^L(h,s(g)),
\end{align*}
as claimed.
\end{proof}

\begin{theorem}
Let $H$ be a dual quasi-Hopf algebra. Then $H_0$ is braided cocommutative as an coalgebra in $^H_H\mathcal{YD}$, that is, for all $h\in H$,
$$h_{\un{1}}\o h_{\un{2}}=h_{\un{1}(-1)}\tr h_{\un{2}}\o h_{\un{1}(0)}.$$
\end{theorem}

\begin{proof}
For all $h\in H$,
\begin{align*}
&h_{\un{1}(-1)}\tr h_{\un{2}}\o h_{\un{1}(0)}\\
&=h_{\un{1}1}s(h_{\un{1}3})\tr h_{\un{2}}\o h_{\un{1}2}\\
&=\sigma(h_1,s(h_5),h_9s(h_{11}))\sigma^{-1}(s(h_6),h_8,s(h_{12}))\a(h_7)h_{2}s(h_{4})\tr h_{10}\o h_{3}\\
&=\underline{\sigma(h_1,s(h_{17}),h_{21}s(h_{29}))}\sigma^{-1}(s(h_{18}),h_{20},s(h_{30}))\\
&\quad \underline{\sigma(h_2s(h_{16}),h_{22},s(h_{28}))}\sigma((h_3s(h_{15}))h_{23},s(h_{26})s(h_6s(h_{12})),h_8s(h_{10}))\\
&\quad g(h_5s(h_{13}),h_{25})q^R(s(h_{27}),s(h_7s(h_{11})))\a(h_{19})(h_{4}s(h_{14}))h_{24}\o h_9\\
&\stackrel{(\ref{1b})}{=}\sigma(s(h_{17(2)}),h_{21},s(h_{29}))\sigma(h_1,s(h_{17(1)})h_{22(1)},s(h_{28}))\sigma^{-1}(s(h_{18}),h_{20},s(h_{30}))\\
&\quad \sigma(h_2,s(h_{16}),h_{22(2)})\sigma((h_3s(h_{15}))h_{23},s(h_{26})s(h_6s(h_{12})),h_8s(h_{10}))\\
&\quad g(h_5s(h_{13}),h_{25})q^R(s(h_{27}),s(h_7s(h_{11})))\a(h_{19})(h_{4}s(h_{14}))h_{24}\o h_9\\
&\stackrel{(\ref{1e})}{=}\sigma(h_1,s(h_{15}),h_{17})\sigma((h_2s(h_{14}))h_{18},s(h_{21})s(h_5s(h_{11})),h_7s(h_{9}))\\
&\quad g(h_4s(h_{12}),h_{20})q^R(s(h_{22}),s(h_6s(h_{10})))\a(h_{16})(h_{3}s(h_{13}))h_{19}\o h_8\\
&\stackrel{(\ref{1h})}{=}\sigma(h_1,s(h_{15}),h_{17})\sigma((h_2s(h_{14}))h_{18},s((h_4s(h_{12}))h_{20}),h_7s(h_{9}))\\
&\quad g(h_5s(h_{11}),h_{21})q^R(s(h_{22}),s(h_6s(h_{10})))\a(h_{16})(h_{3}s(h_{13}))h_{19}\o h_8\\
&\stackrel{(\ref{1a})(\ref{1e})}{=}\sigma(h_3,s(h_{13}),h_{17})\sigma(h_1,s((h_4s(h_{12}))h_{18}),h_7s(h_{9}))\\
&\quad g(h_5s(h_{11}),h_{19})q^R(s(h_{20}),s(h_6s(h_{10})))\a(h_{14})h_{2}\o h_8\\
&\stackrel{(\ref{1a})(\ref{1e})}{=}\underline{\sigma(h_4,s(h_{12}),h_{15})}\sigma(h_1,s(h_3),h_7s(h_{9}))\\
&\quad \underline{g(h_5s(h_{11}),h_{16})}q^R(s(h_{17}),s(h_6s(h_{10})))\a(h_{13})h_{2}\o h_8\\
&\stackrel{(\ref{1h'})}{=}g(h_4,s(h_{15})h_{17})g(s(h_{14}),h_{21})\sigma(s(h_{22}),s^2(h_{13}),s(h_5))f(h_6,s(h_{12}))\\
&\quad \sigma(h_1,s(h_3),h_8s(h_{10}))q^R(s(h_{24}),s(h_7s(h_{11})))\a(h_{16})h_{2}\o h_{9}\\
&\stackrel{(\ref{1e})}{=}\varepsilon(h_1)\underline{g(s(h_{14}),h_{16})}\sigma(s(h_{17}),s^2(h_{13}),s(h_5))f(h_6,s(h_{12}))\\
&\quad \sigma(h_2,s(h_4),h_8s(h_{10}))q^R(s(h_{18}),s(h_7s(h_{11})))\underline{\a(h_{15})}h_{3}\o h_{9}\\
&\stackrel{(\ref{4d})}{=}\varepsilon(h_1)\sigma(s(h_{15}),s^2(h_{13}),s(h_5))f(h_6,s(h_{12}))\\
&\quad \sigma(h_2,s(h_4),h_8s(h_{10}))q^R(s(h_{16}),s(h_7s(h_{11})))\b(s(h_{14}))h_{3}\o h_{9}\\
&=\varepsilon(h_1)\underline{\sigma(s(h_{15}),s^2(h_{13}),s(h_5))}\underline{f(h_7,s(h_{11}))}\\
&\quad \sigma(h_2,s(h_4),h_8s(h_{10}))\underline{q^R(s(h_{16}),s^2(h_{12})s(h_6))}\underline{\b(s(h_{14}))}h_{3}\o h_{9}\\
&\stackrel{(\ref{4e})}{=}\sigma(h_1,s(h_3),h_5s(h_{7}))q^L(h_4,s(h_{8}))h_{2}\o h_{6}\\
&\stackrel{(\ref{4a})}{=}h_{\un{1}}\o h_{\un{2}}.
\end{align*}
The proof is completed.
\end{proof}

\section*{Acknowledgement}

This work was supported by the NSF of China (No. 11901240, 11871301).


\begin{thebibliography}{10}

\bibitem{AP} A. Ardizzoni, A. Pavarin, Bosonization for dual quasi-bialgebras and preantipode, J. Algebra, 390(2013): 126--159.

\bibitem{Ba} A. Balan, Galois extensions for coquasi-Hopf algebras, Comm. Algebra, 38(2010): 1491--1525.

\bibitem{Ba3} A. Balan, Yetter-Drinfeld modules and Galois extensions over coquasi-Hopf algebras, U.P.B. Sci. Bull., Series A, 71(2009): 43--60.

\bibitem{Ba2}A. Balan, Yetter-Drinfeld modules and Galois extensions over coquasi-Hopf algebras, Politehn. Univ. Bucharest Sci. Bull. Ser. AAppl. Math. Phys.71(3)(2009): 43--60(in English,Romanian summary).

\bibitem{BCP} D. Bulacu, S. Caenepeel, F. Panaite, Yetter-Drinfeld modules for quasi-Hopf algebras, Comm. Algebra 34(2006): 1--35.

\bibitem{BCP2} D. Bulacu, S. Caenepeel, F. Panaite, More properties of Yetter-Drinfeld modules over quasi-Hopf algebras, arXiv:math/0311381.

\bibitem{BN1} D. Bulacu, E. Nauwelaerts, Relative Hopf modules for (dual) quasi-Hopf
algebras. J. Algebra, 22(2000): 632--659.

\bibitem{BN3} D. Bulacu, E. Nauwelaerts, Radford's biproduct for quasi-Hopf algebras and bosonization, J.Pure Appl. Algebra 174(2002): 1--42.

\bibitem{BN2} D. Bulacu, E. Nauwelaerts, Dual quasi-Hopf algebra coactions, smash
coproducts and relative Hopf modules, Rev. Roumaine Math. Pures Appl. 47(2003): 415--443.

\bibitem{BT} D. Bulacu, B. Torrecillas, Factorizable quasi-Hopf algebras-applications,
J. Pure Appl. Algebra 194 (2004): 39--84.

\bibitem{D} V. Drinfeld, Quasi-Hopf algebras, Leningrad Math. J. 1(1990): 1419--1457.

\bibitem{F} X. L. Fang, S. H. Wang, New Turaev braided group categories and group corings based on quasi-Hopf group coalgebras, Comm. Algebra 41(2013): 4195--4226.

\bibitem{H} H. L. Huang, Quiver approaches to quasi-Hopf algebras, J. Math. Phys. 50(2009), 043501.

\bibitem{L} J. Q. Li, Dual quasi-Hopf algebras and antipodes, Algebra Colloquium, 13(2006): 111--118.

\bibitem{M1} S. Majid, Tannaka-Kre$\check{i}$n theorem for quasi-Hopf algebras and other results, in: Deformation Theory and Quantum Groups with Applications to Mathematical Physics, Amherst, MA, 1990, in: Contemp.Math., vol.134, Amer. Math. Soc., Providence, RI, 1992, pp.219--232.

\bibitem{M2} S. Majid, Quantum double for quasi-Hopf algebras, Lett. Math. Phys. 45(1998): 1--9.

\bibitem{Zhang} X. H. Zhang, S. H. Wang, New Turaev braided group categories and weak (co)quasi-Turaev group coalgebras, J. Math. Phys. 55(2014), 111702.

\end{thebibliography}
\end{document}